\documentclass[a4paper,12pt]{amsart}
\usepackage{newlfont,amsmath,amsthm,amssymb,paralist}
\usepackage[top=3cm, bottom=2.8cm, left=2.5cm, right=2.5cm]{geometry}
\usepackage{hyperref}

\usepackage{stmaryrd}
\usepackage[normalem]{ulem}
\usepackage{mathrsfs}

\usepackage{braket}
\usepackage{graphicx}
\usepackage{arydshln}
\usepackage{multirow}

\usepackage[all]{xy}

\usepackage[noabbrev]{cleveref}

\newcommand{\delete}[1]{\empty}

\newcommand{\Sym}{{\mathrm{Sym}}}

\newcommand{\bfzero}{{\mathbf{0}}}

\newcommand{\bfE}{{\mathbf{E}}}

\newcommand{\id}{{\mathrm{id}}}

\newcommand{\Gr}{{\mathrm{Gr}\,}}

\newcommand{\rank}{{\mathrm{rank}\,}}

\newcommand{\bC}{{\mathbb C}}

\newcommand{\bN}{{\mathbb N}}

\newcommand{\bR}{{\mathbb R}}

\newcommand{\bZ}{{\mathbb Z}}

\newcommand{\rO}{{\mathrm{O}}}
\newcommand{\rU}{{\mathrm{U}}}

\newcommand{\pr}{{\mathrm{pr}}}
\newcommand{\Sp}{{\mathrm{Sp}}}

\newcommand{\GL}{{\mathrm{GL}}}
\newcommand{\SO}{{\mathrm{SO}}}

\newcommand{\Stab}{{\mathrm{Stab}}}

\newcommand{\frakg}{{\mathfrak{g}}}
\newcommand{\frakh}{{\mathfrak{h}}}

\newcommand{\frakm}{{\mathfrak{m}}}

\newcommand{\frakk}{{\mathfrak k}}

\newcommand{\fraksl}{{\mathfrak{sl}}}
\newcommand{\frakso}{{\mathfrak{so}}}
\newcommand{\fraksp}{{\mathfrak{sp}}}

\newcommand{\frakp}{{\mathfrak p}}
\newcommand{\frakq}{{\mathfrak q}}

\newcommand{\fsp}{{\mathfrak{sp}}}

\newcommand{\spl}{{\mathrm{Sp}}}

\newcommand{\Ind}{{\mathrm{Ind}}}

\newcommand{\wtU}{\widetilde{\mathrm{U}}}
\newcommand{\Hom}{{\mathrm{Hom}}}

\def\wtSp{\widetilde{\mathrm{Sp}}}

\newtheorem{thmA}{Theorem}

\newtheorem*{prop*}{Proposition}

\newtheorem{lemma}[subsection]{Lemma}

\newtheorem{prop}[subsection]{Proposition}
\newtheorem{thm}[subsection]{Theorem}
\newtheorem{cor}[subsection]{Corollary}
\theoremstyle{definition}
\newtheorem{definition}[subsection]{Definition}

\newcommand{\cD}{{\mathcal{D}}}
\newcommand{\cE}{{\mathcal{E}}}
\newcommand{\cR}{{\mathcal{R}}}
\newcommand{\cS}{{\mathcal{S}}}

\newcommand{\wtcH}{{\widetilde{\cH}}}

\newcommand{\calC}{{\mathcal{C}}}

\newcommand{\calH}{{\mathcal{H}}}

\newcommand{\calO}{{\mathcal{O}}}

\newcommand{\AC}{\mathrm{AC}}
\newcommand{\AV}{\mathrm{AV}}

\newcommand{\Spec}{\mathrm{Spec}}

\newcommand{\red}{\color{red}}
\newcommand{\blue}{\color{blue}}

\def\codim{{\mathrm{codim}}}
\def\wtE{\widetilde{E}}
\def\wtG{\widetilde{G}}

\def\wtK{\widetilde{K}}
\def\wtL{\widetilde{L}}

\def\wtrU{\widetilde{{\mathrm{U}}}}

\def\fgg{\mathfrak{g}}

\def\fkk{\mathfrak{k}}
\def\fmm{\mathfrak{m}}
\def\fpp{\mathfrak{p}}
\def\fss{\mathfrak{s}}
\def\fuu{\mathfrak{u}}
\def\fN{\mathfrak{N}}

\def\cH{\mathcal{H}}
\def\cI{\mathcal{I}}
\def\cN{\mathcal{N}}

\def\cO{\mathcal{O}}
\def\cS{\mathcal{S}}
\def\cU{\mathcal{U}}

\def\bcN{{\overline{\cN}}}
\def\bcO{{\overline{\cO}}}

\def\sA{{\mathscr{A}}}
\def\sB{{\mathscr{B}}}
\def\sO{{\mathscr{O}}}
\def\sL{{\mathscr{L}}}
\def\sN{{\mathscr{N}}}
\def\sY{{\mathscr{Y}}}
\def\Ann{{\mathrm{Ann}}}
\def\Supp{{\mathrm{Supp}\,}}

\def\diag{\triangle}

\long\def\delete#1{}

\def\notshow#1{}

\def\cM{{\mathcal{M}}}
\def\bcNp{{\overline{\cN^+}}}
\def\bcNn{{\overline{\cN^-}}}
\def\bcOp{{\overline{\cO'}}}

\def\bcNs{{\overline{\cN_s}}}
\def\bfG{H}

\def\tchi{{\widetilde{\chi}}}
\def\bcOpd{\overline{\cO_d'}}

\def\sspan{{\mathrm{span}}}

\begin{document}

\subjclass{22E46, 22E47}

\title[Associated cycles]{Associated cycles of local theta lifts of
  unitary characters and unitary lowest weight modules}

\author{Hung Yean Loke}

\author{Jia-jun Ma}

\address{Department of Mathematics,
National University of Singapore,
2 Science Drive 2, Singapore 117543}

\email{matlhy@nus.edu.sg, jiajunma@nus.edu.sg}

\author{U-Liang Tang}

\address{School of Mathematics and Science,
Singapore Polytechnic, 500 Dover Road, Singapore 139651}

\email{TANG\underline{ }U\underline{ }LIANG@SP.EDU.SG}

\def\red{\color{red} }
\def\blue{\color{blue} }

\begin{abstract}
  In this paper we first construct natural filtrations on the full
  theta lifts for any real reductive dual pairs.  We will use these
  filtrations to calculate the associated cycles and therefore the
  associated varieties of Harish-Chandra modules of the indefinite
  orthogonal groups which are theta lifts of unitary lowest weight
  modules of the metaplectic double covers of the real symplectic
  groups. We will show that some of these representations are special
  unipotent and satisfy a $K$-type formula in a conjecture of Vogan.
\end{abstract}

\keywords{local theta lifts, nilpotent orbits, associated varieties,
  associated cycles, special unipotent representations.}

\maketitle

\section{Introduction} \label{sec:intro}

In this paper, we first study natural filtrations on the full theta
lifts for any real reductive dual pairs.  We will use these
filtrations to calculate the associated cycles and therefore the
associated varieties of Harish-Chandra modules of the indefinite
orthogonal groups which are theta lifts of unitary lowest weight
modules of the metaplectic double covers of the real symplectic
groups. We will show that some of these representations are special
unipotent and satisfy a $K$-type formula in a conjecture of Vogan in
\cite{VoAss}.

\subsection{}\label{S11}
Let $W_\bR$ be a real symplectic space and $\wtSp(W_\bR)$ be the
metaplectic double cover of the symplectic group $\Sp(W_\bR)$. For
every subgroup $E$ of $\Sp(W_\bR)$, we let $\wtE$ denote its inverse
image in $\wtSp(W_\bR)$.  Let $(G,G')$ be a real reductive dual pair
in $\Sp(W_\bR)$. We fix a maximal Cartan subgroup $\rU$ of
$\Sp(W_\bR)$ such that $K = G\cap\rU$ and $K' = G' \cap \rU$ are
maximal compact subgroups of $G$ and $G'$ respectively.

Let $\fgg = \fkk \oplus \fpp$ and $\fgg' = \fkk' \oplus \fpp'$ denote
the complexified Cartan decompositions of the complex Lie algebras of
$G$ and $G'$ respectively. Let $\rho'$ be an irreducible
$(\fgg',\wtK')$-module. We will recall the definition of its full
theta lift $\Theta(\rho')$ in \eqref{eqHowequotient}. We suppose
that $\Theta(\rho')$ is nonzero. It is a $(\fgg,\wtK)$-module of
finite length.

Let $\fN_{K_\bC}(\fpp^*)$ be the set of nilpotent $K_\bC$-orbits
in~$\fpp^*$. Similarly let $\fN_{K_\bC'}(\fpp'^*)$ be the set of
nilpotent $K_\bC'$-orbits in~$\fpp'^*$.  We will review the definition
of associated cycle $\AC(\rho')$ of the Harish-Chandra module $\rho'$
in Section \ref{sec:def}. This is a formal nonnegative integral sum of
nilpotent $K_\bC'$-orbits in $\fpp'^*$.  We will also recall the
definition of theta lift of nilpotent orbits $\theta :
\fN_{K_\bC'}(\fpp'^*) \rightarrow \fN_{K_\bC}(\fpp^*)$ in Section
\ref{S14}. This in turn defines $\theta(\AC(\rho'))$.

This paper is motivated by \cite{P}, \cite{Y} and \cite{NZ}. Our hope
is to prove that for dual pair $(G,G')$ we have the identity
\begin{equation} \label{eqnn13} \AC(\Theta(\rho')) =
  \theta(\AC(\rho'^*))
\end{equation}
for any irreducible Harish-Chandra module $\rho'$ of $G'$ such that
$\Theta(\rho')$ is nonzero.  This identity is certainly false in
general. In a recent paper \cite{LM}, the first two authors prove that
\eqref{eqnn13} holds for a type~I reductive dual pair $(G,G')$ in
stable range where $G'$ is the smaller member.  The proof uses a
natural filtration on $\Theta(\rho')$ and assumes some of its
properties.  The first objective of this paper is to provide the
proofs of these properties. See Section \ref{sec:fil}.

The second objective of this paper is to provide evidence that the
identity \eqref{eqnn13} extends beyond the stable range. We will work
with the dual pair $(G,G') = (\rO(p,q), \Sp(2n,\bR))$. Using a more
detail analysis of the geometry of the null cones and moment maps, we
are able to prove in this paper that for certain range outside the
stable range, \eqref{eqnn13} continues to hold if~$\rho'$ is a unitary
lowest weight module. See Theorem \ref{thm:LL}. This extends and gives
a shorter proof to a main result of \cite{NZ}.  We can also compute
the associated cycles of certain $\Theta(\rho')$ when \eqref{eqnn13}
fails. See Theorem \ref{TC}.

Although we only work with the orthogonal symplectic dual pairs, most
of our results extends to the dual pairs $(\rU(p,q), \rU(n_1,n_2))$
and $(\Sp(2p,2q), \rO^*(2n))$ too. See Section \ref{S7}. We hope that
our investigation would shed light on how to extend \eqref{eqnn13} in
general beyond the stable range.

\medskip

For the rest of this section, we will describe and state our main theorems.

\subsection{Associated varieties and associated cycles}\label{sec:def}
First we briefly review the definitions of associated varieties,
associated cycles and other related invariants of a $(\fgg,
K)$-module. See Section~2 in \cite{VoAss} for details.

Let $\varrho$ be a $(\fgg,K)$-module of finite length and let $0
\subset F_0 \subset \cdots\subset F_j \subset F_{j+1} \subset \cdots$
be a good filtration of $\varrho$.  Then $\Gr\varrho = \bigoplus
F_{j}/F_{j-1}$ is a finitely generated $(\cS(\fpp),K_\bC)$-module where
$\cS(\fpp)$ is the symmetric algebra on $\fpp$ and $K_\bC$ is a
complexification of $K$.

Let $\sA$ be the associated $K_\bC$-equivariant coherent sheaf of $\Gr
\varrho$ on $\fpp^* = \Spec(\cS(\fpp))$.  The {\it associated variety}
of $\varrho$ is defined to be $\AV(\varrho): = \Supp(\sA)$ in
$\fpp^*$.  Its dimension is called the {\it Gelfand-Kirillov
  dimension} of $\varrho$.  Let $N(\fpp^*) := \Set{x \in \fpp^*|0\in
  \overline{K_\bC \cdot x}}$ be the nilpotent cone in
$\fpp^*$. Alternatively, we may identify $\fpp^* \simeq \fpp$ using
the Killing form and $N(\fpp^*)$ is defined as the subset of $\fpp^*$
which corresponds to the set of nilpotent elements $N(\fpp)$
in~$\fpp$. It is well known that $\AV(\varrho)$ is a closed
$K_\bC$-invariant subset of $N(\fpp^*)$.

Let $\AV(\varrho) = \bigcup_{j=1}^r \overline{\cO_j}$ such that
$\cO_j$ are distinct open $K_\bC$-orbits in $\AV(\varrho)$.  By
Lemma~2.11 in~\cite{VoAss}, there is a finite
$(\cS(\fpp),K_\bC)$-invariant filtration $0\subset \sA_0 \subset
\cdots \subset \sA_l \subset \cdots \subset \sA_n = \sA$ of $\sA$ such
that $\sA_l/\sA_{l-1}$ is generically reduced on each
$\overline{\cO_j}$.  For a closed point $x_j\in \cO_j$, let
$i_{x_j}\colon \{ x_j \} \hookrightarrow \cO_j$ be the natural
inclusion map and let $K_{x_j} = \Stab_{K_\bC}(x_j)$ be the stabilizer
of $x_j$ in $K_\bC$.  Now
\[
\chi_{x_j}:= \textstyle\bigoplus_l (i_{x_j})^* (\sA_l/\sA_{l-1})
\]
is a finite dimensional rational representation of $K_{x_j}$.  We call
$\chi_{x_j}$ an {\it isotropy representation} of $\varrho$ at
$x_j$. Its image $[\chi_{x_j}]$ in the Grothendieck group of finite
dimensional rational $K_x$-modules is called the {\it isotropy
  character} of $\varrho$ at $x_j$.  The isotropy representation
depends on the filtration.  On the other hand the isotropy character
is an invariant, i.e.  it is independent of the filtration.

We define the {\it multiplicity of $\varrho$ at $\cO_j$} to be
 \[
m(\cO_j,\varrho) =
\dim_\bC\chi_{x_j}\]
 and the {\it associated cycle} of $\varrho$ to be
\[
\AC(\varrho) = \sum_{j=1}^r m(\cO_j,\varrho) [\overline{\cO_j}].
\]

In \Cref{sec:fil}, we will study the filtrations of local theta lifts
generated by the joint harmonics. 

\subsection{Local theta correspondence}\label{sec:def.theta}
Let $(G,G')$ be a real reductive dual pair in $\Sp(W_\bR)$.  We
recall some basic facts of theta correspondences.  Let $\fsp(W_\bR)
\otimes \bC$ denote the complex Lie algebra of $\Sp(W_\bR)$ and 
let $\sY$ be the Fock model (i.e. $(\fsp(W_\bR) \otimes
\bC, \wtrU)$-module)) of the oscillator representation. Let $\rho'$ be
a genuine $(\fgg', \wtK')$-module.  By~\cite{H2},
\begin{equation} \label{eqHowequotient}
\sY / (\textstyle \bigcap_{T \in \Hom_{\fgg,\wtK}(\sY, \rho')} \ker T)
\simeq \rho' \otimes \Theta(\rho')
\end{equation}
where $\Theta(\rho')$ is a $(\fgg,\wtK)$-module called the {\it full
  (local) theta lift} of~$\rho'$. Theorem~2.1 in~\cite{H2} states that
if $\Theta(\rho') \neq 0$, then $\Theta(\rho')$ is a
$(\fgg,\wtK)$-module of finite length with an infinitesimal
character and it has an unique irreducible quotient~$\theta(\rho')$
called the {\it (local) theta lift} of $\rho'$.  We set $\theta(\rho')
= 0$ if $\Theta(\rho') = 0$. Let $\cR(\fgg',\wtK';\sY)$ denote the set
of irreducible $(\frakg',\wtK')$-modules such that $\Theta(\rho')\neq
0$. Then $\rho' \mapsto \theta(\rho')$ is bijection from
$\cR(\fgg',\wtK';\sY)$ to $\cR(\fgg,\wtK;\sY)$.  Similarly, we could
define the theta lifting from $\cR(\fgg, \wtK ; \sY)$ to $\cR(\fgg',
\wtK'; \sY)$.

\subsection{Theta lifts of orbits and cycles} \label{S14}
We recall the complexified Cartan decompositions $\fgg = \fkk \oplus
\fpp$ and $\fgg' = \fkk' \oplus \fpp'$.  There are two moment maps
\begin{equation} \label{eq2} \xymatrix{ \fpp'^* & \ar[l]_{\psi} W
    \ar[r]^{\phi} & \fpp^*.} 
\end{equation}
We also recall the set of nilpotent elements $N(\fpp^*)$ in
$\fpp^*$ and the set of nilpotent $K_\bC$-orbits $\fN_{K_\bC}(\fpp^*)$
in~$\fpp^*$. Similarly we have $N(\fpp'^*)$ and $\fN_{K'_\bC}(\fpp'^*)$.

\begin{definition} 
\begin{asparaenum}[(i)]
\item For a $K_\bC'$-invariant closed subset $S'$ of~$\frakp'^*$,
  we define {\it the theta lift of $S'$} to be
\[
\theta(S') = \theta(S';G',G) := \phi(\psi^{-1} (S')).
\]
It is a $K_\bC$-invariant closed subset of $\fpp^*$.  If $S' \subseteq
N(\fpp'^*)$, then $\theta(S') \subseteq N(\fpp^*)$.

\item If $S' = \overline{\cO'} \subset N(\fpp'^*)$ is the closure of a $K_\bC'$-orbit
  $\cO'$  and $\theta(S') = \overline{\cO}$ is the closure of a
  $K_\bC$-orbit~$\cO$, then we denote $\cO$ by $\theta(\cO') =
  \theta(\cO'; G',G)$.

  Conversely for a closed $K_\bC$-invariant subset $S$ of $N(\fpp^*)$,
  we define $\theta(S) = \theta(S;G,G') := \psi(\phi^{-1} (S))$ which
  is a closed $K_\bC'$-invariant subset of $N(\fpp'^*)$.  When
  $\theta(\bcO) = \overline{\cO'}$, we define $\theta(\cO) = \cO'$.

\item We extend theta lifts of nilpotent orbits to cycles linearly.
  More precisely, we define $\theta(\sum_{j} m_j[\overline{\cO'_j}]) =
  \sum_{j} m_j[\overline{\theta(\cO'_j)}]$ if every $\cO'_j$ admits a
  theta lift.
\end{asparaenum}
\end{definition}

\subsection{} \label{S15} From Section \ref{S2} onwards, we specialize
to the dual pair $(G^{p,q},G') = (\rO(p,q), \spl(2n,\bR))$ in
$\Sp(W_\bR) = \Sp(2n(p+q),\bR)$. Choosing a maximal compact subgroup
$\rU$ as in \Cref{S11}, we set $K^{p,q} = G^{p,q} \cap \rU
\cong \rO(p) \times \rO(q)$ and $K'^{n}= G' \cap \rU \cong \rU(n)$ to
be the maximal compact subgroups of $G$ and $G'$ respectively. We also
denote $\rO(p)$ by $K^p$.

\medskip

A Harish-Chandra module of $\wtG$ is called {\it genuine} if it does
not descend to a Harish-Chandra module of~$G$. We will introduce some
genuine Harish-Chandra modules in this paper.
\begin{asparaenum}[(a)]
\item $\theta^{p,q}(\sigma')$ : Let $\sigma'$ be the genuine one
  dimensional character of $\wtG'=\wtSp(2n,\bR)$ such that
  $\sigma'|_{\fgg'}$ is trivial.  It exists if and only if $\wtG'$ is
  a split double cover of $G'$, i.e. $p+q$ is even. We say that the
  dual pair $(G^{p,q}, G')$ is under the {\it stable range} if:
\begin{equation}
\label{eq:range.C}
\min(p,q) \geq 2n \quad \mbox{ and } \quad \max(p,q) > 2n.
\end{equation}
Let $\theta^{p,q}(\sigma') := \theta(\sigma')$. Then, in stable range,
it is a nonzero unitarizable genuine Harish-Chandra module of
$\wtG^{p,q}$ (c.f. \cite{Lo,ZH}).
\item $L(\mu')$: \label{item:def.L} Let $\mu$ be a genuine
  $\wtG^{t,0}$-module.  Let $L(\mu'):=\theta(\mu)$ be the
  $(\fgg',\wtK')$-module, which is the theta lift of $\mu$.  It is
  well known that $L(\mu')$ is a unitary lowest module and it is also
  the full theta lift of $\mu$. Here $\mu'$ denotes the lowest
  $\wtK'$-type.

\item $\theta^{p,q}(L(\mu'))$: We suppose that $p+q+t$ is an even
  integer. Then the double cover $\wtG'$ for the dual pairs
  $(G^{p,q},G')$ and $(G^{t,0},G')$ are isomorphic.  Let
  $\theta^{p,q}(L(\mu'))$ be the $(\fgg,\wtK^{p,q})$-module which is
  the theta lift of the $L(\mu')$.
\end{asparaenum}





\subsection{}
\label{sec:lift.char} 
Before we state our main results, we have to describe the theta lifts
of certain orbits.  Since all groups appearring here are classical, we
will use signed Young diagram to parametrize nilpotent $K_\bC$-orbits
in $\fN_{K_\bC}(\fpp^*)$. More precisely, let $\fgg_0$ denote the real
Lie algebra of a classical group $G$ and let $\fN_{G}(\fgg_0)$ denote
the set of nilpotent $G$-orbits in~$\fgg_0$.  Then $\fN_{G}(\fgg_0)$
is parametrized by signed Young diagrams or signed partitions(c.f. Section~9.3
\cite{CM}). As before we identify $\fpp^* \simeq \fpp$ using the
Killng form. Then the Kostant-Sekiguchi correspondence identifies
$\fN_{G}(\fgg_0)$ with the set of nilpotent $K_\bC$-orbits
$\fN_{K_\bC}(\fpp) \simeq
\fN_{K_\bC}(\fpp^*)$. (c.f. Theorem~9.5.1~\cite{CM}).

\begin{asparaenum}[(a)]
\item $\cO'_d$: We consider the compact dual pair $G^{0,t} \times
  G'$. Let $0$ denote the zero orbit of $\fpp = 0$.  Let $\cO'_d
  :=\theta(0; G^{0,t}, G') $ where $d = \min\set{t,n}$.  Here
  $\cO'_d$ only depends on $d$. Indeed let $\fgg' = \frakk' \oplus
  \frakp'^+ \oplus \frakp'^-$ denote the complexified Cartan
  decomposition for the Lie algebra of the Hermitian symmetric group
  $G'$.  Then $\cO'_d$ is the $K'_\bC$-orbit in $\fpp^-$ generated
  by a sum of $d$ strongly orthogonal non-compact long roots in
  $\fpp^-$.  In terms of partitions, we have
\[
\calO_d' = 2_-^d 1_+^{n-d} 1_-^{n-d}.
\]


\item $\cO_{p,q,t}$: \label{item:O.b} Suppose $2n \leq
  \min\set{p,q+t}$. Let $d = \min(t, n)$.  By \cite{Ohta} or
  \cite{DKP}, $\theta(\overline{\cO_d'}; G',G^{p,q})$ is the closure
  of a single $K_\bC$-orbit $\cO_{p,q,t}$. In terms of partitions,
\begin{equation} \label{eqP16}
\cO_{p,q,t} = \theta(\cO_d';G',G^{p,q}) = \theta(\theta(0)) = 
  3_+^{d} 2_{+}^{n-d} 2_-^{n-d} 1_{+}^{p-2n}1_{-}^{q+d-2n}.
\end{equation}
Note that the Young diagram of $\cO_{p,q,t}$ is obtained by adding a
column to the left of the Young diagram of $\cO_d'$.

\item $\cO_{p,q}$: Suppose $t=0$ in the definition of $\cO_{p,q,t}$, we set
\begin{equation} \label{eqn7}
\cO_{p,q} := \cO_{p,q,0} = \theta(0';G',G^{p,q}) = 2_{+}^{n} 2_-^{n}
1_{+}^{p-2n}1_{-}^{q-2n}.
\end{equation}
\end{asparaenum}
In \cite{CM}, the above orbits are denoted by $\cO'_{d} = 2_{-}^d
1_{+}^{2n-2d}$ and $\cO_{p,q,t} = 3_+^{d} 2_{+}^{2n-2d}
1_{+}^{p-2n}1_{-}^{q+t+d-2n}$.


\subsection{Theta lifts of unitary characters}
We first recover the following theorem which is known to the experts.

\begin{thmA}\label{thm:ACchar}

\begin{asparaenum}[(i)]
  Let $(G,G') = (G^{p,q},G') = (\rO(p,q), \spl(2n,\bR))$ be the dual
  pair in the stable range defined by \eqref{eq:range.C} and $p+q$ is
  even integer. Let $\sigma'$ be the genuine unitary character
  of~$\wtG'$. Let $\cO_{p,q} = \theta(0'; G', G^{p,q})= 2_{+}^n
  2_-^n 1_{+}^{p-2n}1_{-}^{q-2n}$ as in \eqref{eqn7}.
\item   Then 
$
\AV(\theta(\sigma')) = \overline{\cO_{p,q}}$ and $\AC(\theta(\sigma'))
= 1 \, [\overline{\cO_{p,q}}]$.
\item Let $x \in \cO_{p,q}$ and let $K_x = \Stab_{K_\bC}(x)$ be the
  stabilizer of $x$ in $K_\bC$. Then
\[
K_x \simeq \Set{ (p_1,p_2) \in P_{p,n} \times P_{q,n} | \beta_{p,n}(p_1) =
  \beta_{q,n}(p_2) }.
\]
Here $P_{p,n} \cong (\GL(n,\bC) \times \rO(p-2n,\bC)) \ltimes N$
denote the maximal parabolic subgroup of $K_\bC^p$ which stabilizes an
$n$-dimensional isotropic subspace of $\bC^p$ and $\beta_{p,n} :
P_{p,n} \rightarrow \GL(n,\bC)$ denotes the quotient map.

\item Let $\chi_x \colon K_x \rightarrow \bC$ denote the character
  $(p_1 ,p_2) \mapsto \det(\beta_{p,n}(p_1))^{\frac{p-q}{2}}$. Then
  the isotropy representation of $\theta(\rho)$ at $x$ is $\tchi_x =
  \varsigma|_{\wtK}\otimes \chi_{x}$ where $\varsigma$ is the
    minimal $\rU$-type of the oscillator representation $\sY$
    (see \Cref{S24}).  Moreover, we have
\begin{equation}\label{eq:Yang}
\theta(\sigma')|_{\wtK} = \varsigma|_{\wtK} \otimes\Ind_{K_x}^{K_\bC}
\chi_x = \Ind_{\wtK_x}^{\wtK_\bC} \tchi_x.
\end{equation}

\end{asparaenum}
\end{thmA}

Part (i) is a result of \cite{NZ}.  If $n = 2p$, then it is
  also a result of \cite{T}. Equation~\eqref{eq:Yang} is part of Yang's thesis \cite{Y}.  

  Our contribution is to recognize and construct a natural
  $K$-equivariant invertible sheaf on $\calO_{p,q}$.  We add that such
  a sheaf is implicit in~\cite{NZ}. The merit of this is that we could
  now bypass the $K$-types or Hilbert polynomials computations in the
  previous work alluded above and prove (i) and (iii) more
  conceptually and efficiently. We remark that Yamashita has used the
  concept of coherent sheaves to compute the associated cycles of
  unitary lowest weight modules \cite{Ya}.

  Most of \Cref{thm:ACchar} is a special case of \cite{LM}.  One main
  reason for introducing theta lifts of unitary characters is that
  $\theta^{p,q}(L(\mu'))$ could be obtained by taking covariants of
  $\theta^{p,q+t}(\sigma')$ (see \Cref{P24}).  This will allow us to
  compute the associated cycles and associated varieties of
  $\theta^{p,q}(L(\mu'))$.

\subsection{}
We assume that $(G^{p,q+t},G')$ is in the stable range and $p+q+t$
is an even integer. Equivalently we have
\begin{equation} \label{eq16}
p+q+t \mbox{ is even},  \quad 
\min(p,q+t) \geq 2n \quad \mbox{and} \quad \max(p,q+t) > 2n.
\end{equation}
Let $L(\mu') = \theta(\mu)$ denote the unitary
lowest module of $\wtG'$ which is the theta lift of $\mu \in
\cR(\wtG^t; \sY_2)$ as in \Cref{S15}. We denote the dual
  Harish-Chandra module of $L(\mu)$ by $L(\mu)^*$. By \cite{Ya}, $\AV(L(\mu')^*) =
\overline{\cO'_d} = \theta(0,G^{0,t},G')$ where $d = \min \Set{ t, n }$. 

In order to describe the associated cycles of $\theta^{p,q}(L(\mu'))$,
we have to divide into two cases: When $q \geq n$, we will denote this
as Case I. When $q < n$, we will denote this as Case II.  The
geometries of the moment maps in these two cases differ significantly.

\subsubsection{} First we describe the main result for Case I.

\begin{thmA} \label{thm:LL} Suppose that $(G^{p,q+t},G')$ is in the
  stable range satisfying \eqref{eq16} and $q \geq n$.  Let $L(\mu') =
  \theta(\mu)$ be a non-zero irreducible unitarizable lowest module of
  $\wtG'$.

\begin{enumerate}[(i)]
\item The theta lift $\theta^{p,q}(L(\mu'))$ is nonzero. 

\item Let $d = \min\set{t,n}$. Then 
\[
\AV(\theta^{p,q}(L(\mu'))) = \overline{\cO_{p,q,t}} =
\theta(\overline{\cO'_d}) = \theta(\AV(L(\mu')^*)).
\]

\item We fix a closed point $x\in \cO_{p,q,t}$ and a closed point
  $x'\in \cO'_d$. Let $\tchi_{x'}$ be the
  isotropy representation of $L(\mu')^*$ at $x'\in \cO'_d$ (see
  \Cref{sec:AC.L}). Then there is a map $\beta\colon K_{x}\to K'_{x'}$
  such that the isotropy representation of $\Gr\theta(L(\mu'))$ at
  $x$ is
\[
\tchi_{x} = \varsigma|_{\wtK} \otimes(\varsigma|_{\wtK'}\otimes
\tchi_{x'})\circ \beta. 
\]

\item We have
\begin{equation} \label{eq:AC.I}
\AC(\theta(L(\mu'))) = (\dim\tchi_x) [\overline{\cO_{p,q,t}}] = (\dim\tchi_{x'})
[\overline{\theta(\cO_d)}] =  \theta(\AC(L(\mu')^*)).
\end{equation}
\end{enumerate}
\end{thmA}

The proof of the above theorem is given in \Cref{S57}. Part (iv) extends the
main result in~\cite{NZ} where $p,q > 2n$ is considered.

\subsubsection{} Next we describe the main result for Case II. 

\begin{thmA} \label{TC} Suppose that $(G^{p,q+t},G')$ is in the
  stable range satisfying \eqref{eq16} and $q < n$.  Let $L(\mu') =
  \theta(\mu)$ be an irreducible unitarizable lowest module of
  $\wtG'$. We fix a closed point $x\in \cO_{p,q,t}$ and a closed
  point $x'\in \cO'_d$ where $d = \min\set{t,n}$.

\begin{enumerate}[(i)]
\item  Let $L$ be the Levi part of a Levi decomposition of $K_{x}$, then $L
  \cong K_\bC^{q}\times K_\bC^{p-2q}$.

\item Let $\tchi_x$ be the isotropy
  representation of $\theta^{p,q}(L(\mu')^*)$ at $x \in
  \cO_{p,q,t}$. Then
\[
  \tchi_{x}|_{\wtL} = 
  (\theta^{p-2q,t-q}(\sigma')|_{\wtK_\bC^{p-2q}\times \wtK_\bC^{t-q}}\otimes
  \mu|_{\wtK_\bC^{t-q}\times \wtK_\bC^{q}})^{K_\bC^{t-q}} \otimes \varsigma^2|_{K_\bC^q}.
\]

\item The lift $\theta^{p,q}(L(\mu'))$ is nonzero if and only if
  $\tchi_x$ is nonzero.

\item If $\theta^{p,q}(L(\mu'))\neq 0$, then
\[
  \AV(\theta^{p,q}(L(\mu'))) = \overline{\cO_{p,q,t}} =
  \theta(\overline{\cO'_d}) = \theta(\AV(L(\mu')^*))
\]
and
\[
\AC(\theta^{p,q}(L(\mu'))) = (\dim \tchi_{x}) [\bcO].
\]
\end{enumerate}
\end{thmA}

The proof of the above theorem is given in \Cref{S58}.  In general
$\dim \tchi_x$ is not equal to the dimension of the isotropy character
$\tchi_{x'}$ of $\Gr(L(\mu'))$ so \eqref{eq:AC.I} is usually invalid for
Case~II.

\subsection{}
Let $\fgg^{p,q+t} = \fkk_H \oplus \fpp_H$ denote the complexified
Cartan decomposition.  The inclusion $\fgg \rightarrow \fgg^{p,p+t}$
includes a projection map $\pr_H : \fpp_H^* \rightarrow \fpp^*$.  By
\Cref{P24}, $\theta^{p,q}(L(\mu'))$ occurs discretely as a submodule
in $\theta^{p,q+t}(\sigma')$. A general theory of \cite{Kob} gives
$\pr_H(\AV(\theta^{p,q+t}(\sigma')) \supseteq
\AV(\theta^{p,q}(L(\mu'))$.  For the representations considered in
Theorems \ref{thm:LL} and \ref{TC}, the containment is in fact an
equality.  See Lemma~\ref{lem:orbit.O1}

\subsection{}
In both Cases I and II above, we have the following theorem on the
$K$-spectrum of $\theta^{p,q}(L(\mu'))$.

\begin{thmA}\label{thm:KS}
  Suppose $p,q,t$ satisfies \eqref{eq16} and $\min\set{q,t}<n$, then
\[
\theta^{p,q}(L(\mu'))|_{\wtK} = \Ind_{\wtK_{x}}^{\wtK_\bC}\tchi_x.
\]
\end{thmA}

The proof is given in Section \ref{sec:kspec} and it is a consequence
of \Cref{T61}.

\smallskip

Motivated by geometric quantization of orbits, Vogan defined an {\it
  admissible} isotropy representation in Definition 7.13
in~\cite{VoAss}.  It is not difficult to see that
$\theta^{p,q}(L(\mu'))$ for $q \geq n \geq t$ and $\dim \mu = 1$ has
admissible isotropy representation.  Then by Theorems
\ref{thm:ACchar}(iii) and \ref{thm:KS}, these representations satisfy
Vogan's Conjecture 12.1 in \cite{VoAss}. Such modules are candidates
for the conjectured unipotent $(\frakg,\wtK)$-modules attached to the
orbits $\cO_{p,q,t}$. In \Cref{S6}, we will show that
$\theta^{p,q}(L(\mu'))$ is a special unipotent representation in the
sense of \cite{BV}.

\medskip

A first draft of this paper was written before \cite{LM}.  The
  current paper is a major revision where we incorporate ideas from
  \cite{LM}, bypass the $K$-types and asymptote calculations, and give
  more geometric and conceptual proofs to Theorems \ref{thm:ACchar} to
  \ref{thm:KS} above.

\subsection*{Notation}
In this paper, all varieties and schemes are defined over $\bC$.  We
will denote the ring of regular functions on a variety or scheme $X$
by $\bC[X]$. For a real Lie group $K$, its complexification is
  denoted by $K_\bC$.

\subsection*{Acknowledgment}
The first author is supported by an National University of Singapore
grant R-146-000-131-112. We thank K. Nishiyama for his comments on
associated cycles.

\section{Natural filtrations of theta lifts} \label{sec:fil}

In this section we will construct good filtrations using local theta
correspondences. Unless otherwise stated, all Lie algebras are complex
Lie algebras. 

We let $(G,G')$ be an arbitrary reductive dual pair in
$\spl(W_\bR)$. We do not assume that they are in stable range. Let $K$
and $K'$ denote the maximal compact subgroups of $G$ and $G'$
respectively.  We set $\frakg = \frakk \oplus \frakp$ and $\frakg' =
\frakk' \oplus \frakp'$ to be the complexified Cartan decomposition of
the Lie algebras of $G$ and $G'$ respectively. A tilde above a group
will denote an appropriate double cover which is usually clear from
the content.

\medskip

First we review Section 3.3 in Nishiyama and Zhu \cite{NZ}.

We recall that the Fock model of the oscillator representation of
$\wtSp(W_\bR)$ is realized on the Fock space $\sY \cong \bC[W]$ of
complex polynomials on the complex vector space $W$. We follow Howe's
notation~\cite{H2} about diamond dual pairs.  Let $(M,K')$ be the dual
pair in $\spl(W_\bR,\bR)$.  Let~$\frakk_M \oplus (\frakm^{(2,0)}
\oplus \frakm^{(0,2)})$ denote the Cartan decomposition of the
complexified Lie algebra of $M$ such that $\frakm^{(2,0)}$ acts by
multiplication of $K'$-invariant quadratic polynomials on $W$ and
$\frakm^{(0,2)}$ acts by degree two $K'$-invariant differential
operators.

Fact~3 in Howe's paper\cite{H2} states that in
$\frakm$, we have
\begin{equation} \label{eq:8}
\frakm^{(2,0)} \oplus \frakm^{(0,2)} = \frakp \oplus \frakm^{(0,2)}.
\end{equation}
The projection of $\frakp$ to $\frakm^{(2,0)}$ under the decomposition
of the left hand side of \eqref{eq:8} is a $K$-isomorphism.  We
identify $\frakp$ as $\frakm^{(2,0)}$ via
this projection.  

We also have the compact dual pair $(M',K)$ in
$\spl(W_\bR)$. In a similar fashion, we have a Cartan
decomposition
\[
\frakm' = \frakk_{M'} \oplus ({\frakm'}^{(2,0)} \oplus
{\frakm'}^{(0,2)}) \mbox{ \ and \ } \frakp' \cong {\frakm'}^{(2,0)}
\]
where $\frakp'$ is the non-compact part of $\frakg'$.

Let $\sY_b$ be the subspace of complex polynomials in $\bC[W]$ of
degree not greater than~$b$.  Then $\sY = \cup \sY_b$ and gives a
filtration of the Fock model $\sY$. Let $(\rho, V_{\rho})$ be the full
theta lift of an irreducible $(\fgg',\wtK')$-module $(\rho',
V_{\rho'})$. Let $\pi : \sY \twoheadrightarrow V_\rho \otimes V_{\rho'}$
be the natural quotient map. Let $(\tau, V_\tau)$ be a lowest degree
$\wtK$-type of $(\rho, V_\rho)$ of degree $j_0$.  Let $V_\tau \otimes
V_{\tau'}$ be the image of the joint harmonics. We define filtrations
on $V_{\rho}$ and $V_{\rho'}$ by $V_j = \cU_j(\fgg) V_{\tau}$ and
$V_j' := \cU_j(\fgg') V_{\tau'}$ respectively. The filtration $V_j'$
is a good filtration of $V_{\rho'}$ since $\rho'$ is irreducible, and
$V_j$ is a good filtration of $V_{\rho}$ because $V_{\rho} =
\cU(\fgg) V_{\tau}$ due to \cite{H2}.

We view $V_{\rho'^*} = \Hom_\bC(V_{\rho'}, \bC
)_{\wtK'-{\mathrm{finite}}}$.  Let $V_{\tau'^*} \subset V_{\rho'^*}$
be an irreducible $\wtK'$-submodule with type $\tau'^*$ which pairs
perfectly with $V_{\tau'}$.  By Theorem 13 (5) in \cite{He}, the
lowest degree $\wtK'$-type $V_{\tau'}$ has multiplicity one in
$\rho'$.  Hence $V_{\tau'}$ and $V_{\tau'^*}$ are well defined
subspaces in $V_{\rho'}$ and $V_{\rho'^*}$ respectively. We
define a good filtration $V_j'^* := \cU_j(\fgg')V_{\tau'^*}$ on
$V_{\rho'^*}$.
 
Let $l \in V_{\tau'^*}$ be a nonzero linear functional on
$V_{\rho'}$. We consider the composite map

\begin{equation} \label{eq:9} 
\xymatrix@C=4em{\nu \colon \sY
  \ar@{->>}[r]^<>(.5){\pi} &
  V_{\rho} \otimes V_{\rho'}  
 \ar@{->>}[r]^<>(.5){\id \otimes l}& V_{\rho}.
}
\end{equation}

We define $F_j = \nu(\sY_j)$. Since $\nu(\sY_{j_0}) = V_\tau$, we also
have $V_j = \cU_j(\frakg) \nu(\sY_{j_0})$. 

\begin{lemma} \label{L41}
  We have $F_{2j+j_0+1} = F_{2j+j_0}$ and $F_{2j+j_0} = V_j$.
\end{lemma}

The proof is given in \Cref{SL41}. We remark that the first
equality of the above lemma was proved in \cite{NZ}. The second
equality was assumed without proof in that paper.

\Cref{L41} suggests that $\set{V_j}$ is a natural choice of
  filtration. We define $\Gr V_{\rho} = \bigoplus_{j=0}^\infty
V_j/V_{j-1}$ to be the corresponding graded module of $\rho$.

\newcommand{\prKp}{\pr_{K'}} \newcommand{\prpp}{\pr_{\fpp'}}

\subsection{}
For any $(\fgg',K')$-module $V$ and $\fss' \subset \fgg'$, we set
$V_{\fss'} = V/V(\fss')$ and $V_{\fss',K'} = V/V(\fss',K')$ where
$V(\fss') =\sspan\set{ X' v | v \in V, X'\in \fss'} $ and
\[V(\fss',K') = \sspan\set{ X' v,k'v -v| v \in \bfE, X \in \fgg', k \in
  K'}.\]

The next proposition gives another realization of  $\rho =
\Theta(\rho')$ which is crucial for us.

\begin{prop} \label{Prop:Theta}
We have
\[
\Theta(\rho') \cong (V_{\rho'^*} \otimes \sY)_{\fgg',K'} \cong 
\left((V_{\rho'^*} \otimes \sY)_{\fpp'}\right)^{K'}.
\]
\end{prop}

\begin{proof}
  Since $K'$ is compact, the last equality follows if we identify
  $K'$-invariant quotient as $K'$-invariant subspace.

Now we prove the first identity. Let $\sN = \bigcap_{\psi \in
    \Hom_{\fgg',\wtK'}(\sY, \rho')} \ker \psi$ as in
  \eqref{eqHowequotient}. We have
\begin{eqnarray} 
  \lefteqn{  
    \Hom_\bC((V_{\rho'^*} \otimes \sY)_{\fgg',K'},\bC)
    = \Hom_{\fgg',K'}((V_{\rho'^*} \otimes \sY)_{\fgg',K'},\bC)} \nonumber \\
  & = & \Hom_{\fgg',K'}(V_{\rho^*} \otimes \sY, \bC)
  = \Hom_{\fgg',\wtK'}(\sY, \Hom_{\bC}(V_{\rho'^*},\bC)) =
  \Hom_{\fgg',\wtK'}(\sY, V_{\rho'}) \label{eqz12} \\
  & = & \Hom_{\fgg',\wtK'}(\sY/\sN, V_{\rho'}). \label{eqz14}
\end{eqnarray}
The last equality in \eqref{eqz12} follows from the fact that $\sY$ is
$\wtK'$-finite and $\Hom_{\bC}(V_{\rho'^*},\bC)_{\wtK'-\text{finite}} =
V_{\rho'}$.  Starting from \eqref{eqz14}, we reverse the steps by
replacing $\sY$ with $\sY/\sN$ in \eqref{eqHowequotient} and we get
\begin{eqnarray*}
\lefteqn{\Hom_\bC((V_{\rho'^*} \otimes \sY)_{\fgg',K'},\bC)
\cong \Hom_\bC((V_{\rho'^*} \otimes \sY/\sN)_{\fgg',K'},\bC)} \\
& \cong & \Hom_\bC( (V_{\rho'^*} \otimes V_\rho \otimes
V_{\rho'})_{\fgg',K'},\bC)
\cong \Hom_\bC(V_{\rho},\bC).
\end{eqnarray*}
This proves the proposition.
\end{proof}

\subsection{}
Let $\bfE = V_{\rho'^*} \otimes \sY$. We summarize
\Cref{Prop:Theta} in the following diagram
\[
\xymatrix{ \bfE=V_{\rho'^*} \otimes \sY \ar@/^2pc/@{->>}[rr]^{\eta}
  \ar@{->>}[r]^<>(.5){\pr_{\fpp'}} & (V_{\rho'^*} \otimes \sY)_{\fpp'}
  \ar@{->>}[r]^<>(.5){\pr_{K'}}& \left((V_{\rho'^*} \otimes
    \sY)_{\fpp'}\right)^{K'} = V_\rho.  }
\]
where $\pr_{\fpp'}$ is the projection map to the
  $\fpp'$-coinvariant quotient space, $\pr_{K'}$ is the projection
  map to the $K'$-invariant subspace and  $\eta \colon \bfE \to
  \bfE_{\fgg',K'} = V_{\rho}$ is the natural quotient map.

We define a filtration on $\bfE$ by
\begin{equation}\label{eq:wjdec}
\bfE_j = \sum_{2a+b=j} V_a'^* \otimes \sY_b
\end{equation}
and a filtration on $V_{\rho}$ by $E_j = \eta(\bfE_{j_0 + 2j}) =
\eta(\bfE_{j_0 + 2j+1})$.

\begin{lemma} \label{L33}
  The filtrations $E_j$ and $V_j = \cU_j(\fgg) V_\tau$ on
  $(\rho,V_{\rho})$ are the same.
\end{lemma}

\begin{proof}
Since $\nu(\sY_{j_0+j}) = V_{\lfloor j/2 \rfloor}$, we have
$\eta(V_{\tau'^*} \otimes \sY_{j_0+2j}) = V_j$.  Hence $V_j \subseteq
E_j$. On the other hand, for $a + {\lfloor b/2 \rfloor} = j$,
\begin{eqnarray*}
  \eta(V_a'^* \otimes \sY_{j_0+b}) 
  & = & \eta(\cU_a(\fgg') V_{\tau'^*} \otimes \sY_{j_0+b}) \\
  & = & \eta(V_{\tau'^*} \otimes \cU_a(\fgg')\sY_{j_0+b})\
  \ (\text{Since $\fgg'$ acts trivially on the image}) \\
  & \subseteq & \eta(V_{\tau'^*} \otimes \sY_{j_0+b+2a} ) 
  \subseteq V_j.
\end{eqnarray*}
By \eqref{eq:wjdec}, $E_j \subseteq V_j$ and this proves the
lemma.
\end{proof}

Taking the graded module, $\eta$ induces a map
\begin{equation} \label{eq19b}
\xymatrix{ \Gr \rho'^* \otimes_{\cS(\fpp)}  \Gr \sY
  \ar@{->>}[r] & \Gr(\pr_\fpp(\bfE)) \ar@{->>}[r]^<>(.5){\Gr\pr_K}& \Gr
  \rho.}
\end{equation}
We will study \eqref{eq19b} more thoroughly in \cite{LM}. 

\subsection{} \label{S24} We recall that the unitary group $\rU =
\rU(W)$ is a maximal compact subgroup of $\Sp(W_\bR)$. Let
$\fraksp(W_\bR) \otimes \bC = \fuu \oplus \fss$ denote the
complexified Cartan decomposition. Let $\varsigma$ be the minimal one
dimensional $\widetilde{\rU}$-type of the Fock model~$\sY$. For $k \in
\rU$, $\varsigma^{-2}(k)$ is equal to the determinant of the $k$
action on $W$. We extend $\varsigma$ to an
$(\cS(\fss),\widetilde{\rU})$-module where $\fss$ acts trivially. We
will continue to denote this one dimensional module by $\varsigma$.
In this way, $\Gr \sY = \oplus(\sY_{a+1}/\sY_a) \simeq \varsigma
\otimes \bC[W]$ where $\rU$ acts geometrically on $\bC[W]$. Since
$(G,G')$ is a reductive dual pair in $\spl(W_\bR)$, we denote the
restriction of $\varsigma$ as a $(\cS(\fpp),\wtK)$-module by
$\varsigma|_{\wtK}$. Similarly we get a one dimensional
$(\cS(\fpp'),\wtK')$-module $\varsigma|_{\wtK'}$.

Let $A = \varsigma|_{\wtK'} \otimes \Gr V_{\rho'^*}$ and $B =
\varsigma^*|_{\wtK} \otimes \Gr V_{\rho}$. Since $\rho'$ is a genuine
Harish-Chandra module of~$\widetilde{G}$, $A$ is an
$(\cS(\fpp'),K_\bC')$-module. Similarly $B$ is an
$(\cS(\fpp),K_\bC)$-module.

We take $\varsigma$ into account and the fact that $K'$ acts on $A
\otimes \bC[W]$ reductively and preserves the degrees. Then
\eqref{eq19b} gives the following $(\cS(\fpp),K_\bC)$-module morphisms
\begin{equation} \label{eq:gretaiso} \xymatrix{
   A \otimes_{\cS(\fpp')} \bC[W] \ar@{->>}[r]^<>(.5){\eta_1} &
    \left(A \otimes_{\cS(\fpp')} \bC[W] \right)^{K'_\bC}
    \ar@{->>}[r]^<>(.5){\eta_0}&  B.  }
\end{equation}
The merit of introducing $\varsigma$ is that the $\wtK \cdot \wtK'$
action on $\Gr \sY$ descends to a geometric $K_\bC \times K_\bC'$
action on $\bC[W]$.

\begin{prop} \label{L3} Suppose $(G,G')$ is in the stable range where
  $G'$ is the smaller member.  Then~$\eta_0 : (A
  \otimes_{\cS(\fpp')} \bC[W])^{K'_\bC} \stackrel{\sim}{\rightarrow} B$ in
  \eqref{eq:gretaiso} is an isomorphism of $(\cS(\fpp),
  K_\bC)$-modules. 
\end{prop}

We refer to \cite{LM} for a proof.

\section{Local theta lifts of unitary characters} \label{S2}

Throughout this section, we let $(G,G') = (G^{p,q},G') = (\rO(p,q),
\spl(2n,\bR))$ be a dual pair in stable range
(see~\eqref{eq:range.C}). We consider the theta lift $\theta(\sigma')$
of the genuine unitary character $\sigma'$ of $\wtG'$.  We will the
discuss the associated cycle and the isotropy representation of
$\theta(\sigma')$.

\subsection{} Let $\sigma'$ be the genuine unitary character of
$\wtG'$. It exists if and only if the double cover
$\wtG'$ splits over $G'$, i.e. $p+q$ is even. First we recall some
facts in \cite{Lo} about the local theta lift of $\sigma'$ to
$\wtG$. Also see \cite{KO} when $n = 1$. Let~$\fgg$ denote the
complexifixed Lie algebra of $G$. Let $K = K^{p,q} = G^{p,0} \times
G^{0,q}$ be the maximal compact subgroup of $G$.


\begin{prop} \label{P32} Suppose $(G^{p,q},G')$ is in stable range
  where $G'$ is the smaller member satisfying \eqref{eq:range.C} and
  $p+q$ is even. Let $\sigma'$ be the genuine character of
  $\wtG'$. Then $\Theta(\sigma')$ is a nonzero, irreducible and
  unitarizable $(\fgg,\wtK)$-module. In particular $\Theta(\sigma') =
  \theta(\sigma')$.
\end{prop}

\begin{proof} 
  The fact that $\Theta(\sigma')$ is irreducible follows from \cite{ZH}
  or \cite{Lo}. It is also a special case of Theorem A in \cite{LM}.
  The fact that $\theta(\sigma')$ is unitarizable follows from
  \cite{Li}.
\end{proof}

The $\wtK$-types of $\theta(\sigma') = \Theta(\sigma')$ are well
known. For example see \cite{Lo} and \cite{Zhu}. It is also a special
case of Propositions 2.2 and 3.2 in \cite{LM}. We state it as a
proposition below. 

\begin{prop} \label{T22} Let $\wtcH$ denote the $\wtK$-harmonics in
  the Fock model $\sY$ for the dual pair $(G,G')$ in stable
  range. Then as $\wtK$-modules,
\[
  \theta(\sigma')|_{\wtK} = \left(\wtcH\otimes
    \sigma'^*\right)^{\wtK'}.  \qed 
\]
\end{prop}

\subsection{} We refer to \Cref{L3}. If we set $\rho' = \sigma'$ to be
the genuine character of~$\wtG'$, then $A = \varsigma|_{\wtK'} \otimes
\sigma'^*$ is a one dimensional trivial $\cS(\fpp')$-module. As a
representation of $K' \cong \rU(n)$, $A = \det^{\frac{p-q}{2}}$.
Let $B
= \varsigma^*|_{\wtK} \otimes \Gr(\theta(\sigma'))$.  We note that in
this special case, \Cref{L3} is a direct consequence of \Cref{T22}.

We recall the moment maps $\phi : W \rightarrow \fpp^*$ and $\psi : W
\rightarrow \fpp'^*$ in \eqref{eq2}.  We set $\bcN = \psi^{-1}(0')$
and it is called the {\it null cone} in $W$. Let $\cO =\cO_{p,q} =
\theta(0'; G', G)$ be the $K_\bC$-orbit in~\eqref{eqn7}. Then $\cN =
\phi^{-1}(\cO)$ is an open $K_\bC$-orbit in $\bcN$ which we call the
open null cone.  Furthermore, $\psi^*(\fpp')\bC[W]$ is precisely the
radical ideal $I(\bcN)$ of $\bcN$ in $\bC[W]$ (see \cite{GW,H3}). Let $I(\bcO)$ be the
radical ideal of $\bcO$ in $\cS(\fpp)$. We state a corollary of
Proposition \ref{L3}.

\begin{cor}\label{cor:iso.c}
We have an $(\cS(\fpp),K_\bC)$-module isomorphism
\[ 
(\bC[\bcN] \otimes A)^{K'_\bC}\cong B.
\]
Furthermore $I(\bcO) \subset \Ann_{\cS(\fpp)} B$ and $B$ is a
$(\bC[\bcO],K_\bC)$-module.
\end{cor}

\begin{proof}
  If we regard $A$ as the trivial $\cS(\frakp')$-module, then
  $\bC[\bcN] = \bC[W]/I(\bcN) =$ \linebreak $\bC[W]
  \otimes_{\cS(\frakp')} A$.  The identity in the corollary follows
  from Proposition \ref{L3}. We have $\phi^*(I(\bcO)) \subseteq
  I(\bcN)$ so $I(\bcO) \subset \Ann_{\cS(\fpp)} B$.
\end{proof}

\begin{prop} \label{prop:iso.cN}
The natural inclusion $\bC[\bcN]\to \bC[\cN]$ induces an isomorphism of
$(\cS(\fpp), K_\bC)$-modules
\[
B  \cong (\bC[\cN] \otimes A)^{K'_\bC}.
\]
\end{prop}

\begin{proof}
  It suffices to prove that $(\bC[\bcN] \otimes A)^{K'_\bC}\cong
  (\bC[\cN] \otimes A)^{K'_\bC}$ as admissible $K_\bC$-modules.  This
  is verified in \cite{Y}.
\end{proof}

We remark that if $\min(p,q) > 2n$, then the above lemma follows
immediately from the fact that $\bC[\bcN] \cong \bC[\cN]$.  Indeed, in
these cases, $\bcN$ is a normal variety by~\cite{Ko,NOZ} and $\partial
\cN = \bcN - \cN$ has codimension at least~$2$ in~$\bcN$.

\subsection{}
Let $\sB$ be the coherent sheaf associated to the module $B$ on
$\fpp^*$.  Clearly $\AV(\theta(\sigma'))= \Supp\sB \subseteq
\phi(\bcN) = \bcO$.  In order to calculate the isotropy representation
of $\Gr \theta(\sigma')$, we first recall a special case of Corollary
A.5 in~\cite{LM}.

\begin{lemma}
  Fix a $w\in \cN$ such that $\phi(w) = x \in \cO$. Let $K_x$ be the
  stabilizer of $x$ in~$K_\bC$. Then the fiber $F_x = \phi^{-1}(x)
  \cap \bcN$ in $\bcN$ is a single $K_\bC'$-orbit where $K_{\bC}'$
  acts freely. Moreover, there is a (unique) surjective homomorphism
  $\alpha\colon K_x \twoheadrightarrow K'_\bC$ such that
\[
S_w = K_x\times_\alpha K'_\bC = \Set{(k,\alpha(x))\in
  K_x\times K'_\bC}. \qed
\] 
\end{lemma}

For $k \in K_x$, $\alpha(k)$ is the unique element in $K_\bC'$ such
that $(k,\alpha(k))\cdot w = w$.
 
The above definitions are summarized in the
diagram~\eqref{fig:orbchar1} below.
\begin{figure}[h]
  \centering
  \begin{equation}\label{fig:orbchar1}
\xymatrix{
(K_x\times K'_\bC)/S_w \ar@{=}[r]&F_x \ar@{^(->}[r]\ar[d] & \cN  \ar[d]\ar@{^(->}[r]&\bcN
\ar[d]^{\phi|_{\bcN}}\ar@{^(->}[r]&  W\ar[d]^{\phi}\\
& \set{ x } \ar@{^(->}[r]& \cO
\ar@{^(->}[r]^{i_{\cO}} & \overline{\cO}\ar@{^(->}[r] & \fpp^* 
}
  \end{equation}
\end{figure}

\begin{lemma} \label{L46} We assume the notation in the above
  diagram~\eqref{fig:orbchar1}. We also recall the one dimensional
  character $A = \varsigma|_{\wtK_\bC'} \otimes \sigma'^*$ of
  $K_\bC'$. Then we have
\[
i^*_x\sB \cong A \circ \alpha \neq 0
\] 
as a one dimensional character of $K_x$.  Therefore the isotropy
representation $\chi_x$ of $\Gr \theta(\sigma')$ at $x$ is
\[
\chi_x = \varsigma|_{\wtK} \otimes ((\varsigma|_{\wtK'}\otimes
\sigma'^*)\circ \alpha). 
\]
\end{lemma}
\begin{proof}
  Let $\cI_x$ denote the maximal ideal of $x$.  Since $x$ generates an
  open dense orbit $\cO$ in $\overline{\cO}$ and $\phi\colon
  \bcN\twoheadrightarrow \bcO$ is surjective, the scheme theoretic
  fiber $\Spec\left(\bC[\bcN]/\cI_x\bC[\bcN]\right)$ is reduced and
  equals to $F_x$. Then the isotropy representation
\begin{eqnarray*}
i_x^* \sB & = & B/\cI_x B = \left(\bC[\bcN]\otimes
    \sigma'\right)^{K'_\bC}/\left(\cI_x
    \left(\bC[\bcN]\otimes A\right)^{K'_\bC}\right) \\
& = & \left(\bC[\bcN]/\cI_x \bC[\bcN] \otimes
    A\right)^{K'_\bC}
  \quad \text{(by the exactness of taking $K_\bC$-invariants)}\\
& = & \left(\bC[F_x]\otimes \sigma'^*\right)^{K'_\bC}\\
& = & \left(\Ind_{K_x\times_\alpha K'_\bC}^{K_x\times K'_\bC} \bC
    \otimes A\right)^{K'_\bC}
  \quad \text{(because $F_x = (K_x\times K'_\bC)/S_w$)}\\
& \cong & A\circ \alpha.
\end{eqnarray*}
\end{proof}

\begin{proof}[Proof of \Cref{thm:ACchar}]
  By \Cref{cor:iso.c}, $\Supp(\sB)\subseteq \overline{\cO}$.  By
  \Cref{L46}, $\dim_\bC i_x^* \sB = 1$ so $x \in \Supp(\sB)$.  Hence
  $\Supp(B)=\overline{\cO}$. By definition $\AV(\theta(\sigma')) =
  \bcO$ and $\AC(\theta(\sigma')) = 1 \, [\bcO]$.

  Let $\sL$ be the locally free coherent sheaf on $K_\bC/K_x\cong \cO$
  associated with the $K_x$-module $A \circ \alpha$. By \cite{CPS}
  there is an equivalence of categories between certain
  $K_\bC$-equivalent quasi-coherent sheafs and rational
  $K_x$-modules. Applying this to \Cref{L46} gives $i_{\cO}^*\sB \cong
  \sL$. In particular $\sB(\cO) = \sL(\cO) = \Ind_{K_x}^{K_\bC} (A
  \circ \alpha)$.
  
  It remains to show that $B = \sB(\bcO) \hookrightarrow \sB(\cO)$ is
  an isomorphism. We recall that $\phi^{-1}(\cO) = \cN$. Then $B =
  (\bC[\cN] \otimes A)^{K'_\bC} = \sB(\cO)$ by \Cref{prop:iso.cN}.  This
  completes the proof of the theorem.
\end{proof}

\section{Unitary lowest weight modules and its theta lifts}

Throughout this section, we let $(G,G')$ be the dual pair $(\rO(p,q),
\spl(2n,\bR))$ as before. We will discuss some unitary lowest weight
modules of $\wtG'$ and their theta lifts to $\wtG$.

\subsection{} \label{S34a} 
The group~$G'$ is Hermitian symmetric so its has complexified
Cartan decomposition $\fgg' =\fkk'\oplus \fpp'$ and $\fpp' =
\fpp'^+\oplus \fpp'^-$ where $\fpp'^\pm$ are $K_\bC$-invariant abelian
Lie subalgebras of $\fgg'$.

Let $K^t = G^{t,0} \cong \rO(t)$. Let $\sY_2$ be the Fock model of the
oscillator representation for compact dual pair $(K^t,G')$.  Let $\mu
\in \cR(\wtK^t;\sY_2)$. We define
\[
L(\mu') = \theta(\mu) = (\sY_2 \otimes \mu^*)^{K^t}
\]
and $L(\mu')^* = (\sY_2^*\otimes \mu)^{K^t}$. Here $L(\mu')$ is a
lowest weight module with lowest $\wtK'$-type $\mu'$, and $L(\mu')^*$
is its contragredient module.

It is a well known result of \cite{EHW} and \cite{DER} that
all unitarizable lowest weight modules of $G'$ up to unitary
characters are obtained from compact dual pair correspondences. 

Let $W_2$ be the complex space with Hermitian form compatible with
$(G^{0,t},G')$. Its Fock model is $\bC[W_2] \cong
\sY_2^*$. The $\fpp'^-$ action on $\sY_2^*$ is by multiplying degree
two $K_\bC^t$-invariant polynomials. It gives an algebra homomorphism
$\psi_2^*\colon \cS(\fpp'^-) \to \bC[W_2]^{K^t}$ which in turn defines
the moment map
\[
{\psi_2} \colon W_2\to (\fpp'^-)^* = \Spec \, \cS(\fpp'^-).
\]

\newcommand{\redG}{{\red G}}
\newcommand{\bfK}{{K_H}}

\subsection{Associated cycles}\label{sec:AC.L}
In \Cref{sec:fil}, we have a filtration on $L(\mu')^*$ which gives a
graded module $\Gr(L(\mu')^*)$.  Despite the fact that the dual pair
is not in stable range, it is well known that \Cref{L3} extends to the
graded module (for example see \cite{KV} and \cite{Ya}) and we have
\[
B' := \varsigma_2^*|_{\wtK'}\otimes \Gr(L(\mu')^*) =
\varsigma_2^*|_{\wtK'} \otimes (\Gr \sY_2^* \otimes \mu)^{K^t} =
(\bC[W_2]\otimes \tau)^{K^t}
\]
where $\varsigma_2$ is the minimal $\widetilde{\rU}(W_2)$-type of
$\sY_2^*$ and
\begin{equation} \label{eq:def.tau}
\tau =\varsigma_2|_{\wtK^t} \otimes \mu
\end{equation}
We recall $\bcOp = \bcOpd = {\psi_2}(W_2)$ where $d = \min
\set{t,n}$. Hence the graded module $B'$ is a finitely generated
$\bC[\bcOp]$-module.

Let $x'\in \cO'$. We consider following diagram 
\[
\xymatrix{
{\psi_2}^{-1}(x')\ar[r]\ar[d]& W_2 \ar[d]\ar[dr]^{{\psi_2}}\\
\Set{ x' } \ar[r]^{i_{x'}}& \bcOp \ar[r]& (\fpp'^-)^*.
}
\]
Let $\sB'$ be the coherent sheaf on $\bcOp$ associated with the module
$B'$. Let $K_{x'}'$ be the stabilizer of $x'$ in $K_\bC'$.  Then the
isotropy representation of $\sB'$ and $\Gr(L(\mu')^*)$ at $x'$ are
\begin{equation} \label{eq17c} \chi_{x'} = i^*_{x'} \sB' =
  (\bC[{\psi_2}^{-1}(x')]\otimes \tau)^{K^t} \quad \text{and} \quad
  \tchi_{x'} = \varsigma_2|_{\wtK'} \otimes \chi_{x'}
\end{equation}
respectively.  The representation $\chi_{x'} = i^*_{x'}\sB'$ is
calculated in \cite{Ya}. We state the result for pair $(K^t,G') =
(\rO(t),\Sp(2n,\bR))$.

\begin{thm}[\cite{Ya}] \label{T52}
\begin{enumerate}[(i)]
\item The module $L(\mu') = \theta(\mu)$ is nonzero if and only if
  $\chi_{x'}$ is nonzero.

\item Suppose $t\leq n$. Then $K'_{x'} \cong (\rO(t,\bC)\times
  \GL(n-t,\bC)) \ltimes N$ where $N$ is a nilpotent subgroup.  The
  isotropy representation is $\chi_{x'} \cong \tau$ as an
  $\rO(t,\bC)$-module and the other subgroups of
  $K_{x'}'$ act trivially on it.

\item Suppose $t>n$, then we have $K'_{x'}
  \cong \rO(n,\bC)$.  The isotropy representation is $\chi_{x'}
  \cong \tau^{\rO(t-n)}$ as an $\rO(n,\bC)$-module.
\end{enumerate}
\end{thm}

The next theorem follows immediately from the definitions in \Cref{sec:def} 
and the above theorem on the isotropy representations.

\begin{thm}[\cite{NOT,Ya}]
We have $\AV(L(\mu')^*)= \bcOp_d$ and
\[
\AC(L(\mu')^*) = (\dim_{\bC}i^*_{x'}\sB') [\bcOp_d] = \begin{cases}
  (\dim_{\bC} \mu) \,  [\bcOp_{t}] & \text{if } t \leq n\\
  (\dim_{\bC} (\varsigma_2|_{\wtK^t}\otimes \mu)^{\rO(t-n)}) \, 
[\bcOp_n] & \text{if } t > n.
\end{cases} \qed
\]
\end{thm}

\subsection{Theta lift of $L(\mu')$} \label{S34} We consider the dual
pairs $(G, G') = (G^{p,q},G')$ and $(K^t, G') = (G^{t,0},G')$.
Let $(\bfG,G') = (G^{p,q+t},G')$. For reasons which will be clear
later (see \Cref{P24}~(ii)), we assume \eqref{eq16}.  Then $H$ contains $G \times K^t =
G^{p,q} \times G^{0,t}$. We note that $(H,G')$ is in the stable range
but $(G,G')$ could be outside the stable range.  Let $\bfK =
K^{p,q+t}$ denote a maximal compact subgroup of $\bfG$ compatible with
$G$ and $K^t$.  Let $\sY_H$ (resp. $\sY$, $\sY_2$) be Fock model of
the oscillator representation associated to the dual pair $(H,G_0')$
(resp. $(G, G_1')$, $(K^t,G_2')$) where $G_0' = G_1' = G_2' = G'$.
Then as an infinitesimal module of $(\wtG \times \wtG_1') \times
(\wtK^t \times \wtG_2')$,
\begin{equation} \label{eqn20}
\sY_H \cong \sY \otimes \sY_2^*.
\end{equation}
We note that $\wtG_0'$ is a split double cover of $G'$. On the other
hand $\wtG_1' \simeq \wtG_2'$ and it is a split double cover of $G'$
if and only if $t$ is even. Without fear of confusion, we will denote
all the three $\wtG_i'$'s by $\wtG'$.

Let $\theta^{p, q+t} : \cR(\wtG';\sY_H) \rightarrow \cR(\wtG^{p,
  q+t};\sY_H)$ and $\theta^{p,q}: \cR(\wtG';\sY) \rightarrow
\cR(\wtG^{p,q};\sY)$ be the theta lifting maps.  For $\mu \in
\cR(\wtK^t,\sY_2)$, we set $L(\mu') = \theta(\mu)$ to be the unitary
lowest module as defined in \Cref{S34a}. Let $\fgg^{p,q}$ denote the
complex Lie algebra of $G^{p,q}$.

\begin{prop} \label{P24} 
\begin{enumerate}[(i)]
\item For $\mu \in \cR(\wtK^t,\sY_2)$, we have
 \begin{equation} \label{eq7} \Theta^{p,q}(L(\mu')) \cong
    (\Theta^{p, q+t}(\sigma')\otimes \mu)^{K^t}
\end{equation}
as (possibly zero) $(\frakg^{p,q},\wtK^{p,q})$-modules.  

\item
Suppose $p,q,t$ satisifes \eqref{eq16} so that $\Theta^{p,
  q+t}(\sigma')$ is a nonzero and unitarizable\linebreak $(\fgg^{p,
  q+t},\wtK^{p, q+t})$-module by Proposition \ref{P32}. If $\Theta^{p,
  q}(L(\mu'))$ is nonzero, then it is a unitarizable and irreducible
$(\fgg^{p,q}, \wtK^{p,q})$-module. In particular $\Theta^{p,
  q}(L(\mu')) = \theta^{p, q}(L(\mu'))$
\end{enumerate}
\end{prop}

\begin{proof}
  Part (i) is proved in \cite{Lo} as a consequence of a see-saw pair
  argument. In (ii), $\Theta(L(\mu'))$ is unitarizable because it is a
  submodule of the unitarizable module $\Theta(\sigma') \otimes
  \mu$. In particular, $\Theta(L(\mu'))$ is a direct sum of its
  irreducible submodules. On the other hand $\Theta(L(\mu'))$ is a
  full theta lift so it has a unique irreducible quotient module.
  This proves that $\Theta(L(\mu'))$ is irreducible.
\end{proof}

\subsection{Outside stable range} \label{S23} In Propositions
\ref{T22} and \ref{P24}, we have assumed \eqref{eq:range.C} and
\eqref{eq16}. One could easily extend the definitions of
$\Theta(\sigma')$ and $\Theta(L(\mu'))$ beyond these
assumptions. Equation~\eqref{eq7} continues to hold. However both
$\Theta(\sigma')$ and $\Theta(L(\mu'))$ are not necessarily nonzero or
irreducible. We will briefly discuss below. We will denote
$\theta^{p,q+t}(\sigma')$ by $\theta_n^{p, q+t}(\sigma')$ and
$\theta^{p,q}(L(\mu'))$ by~$\theta_n^{p,q}(L(\mu'))$.  We refer the reader
to \cite{Lo} and \cite{LMT} for more details.

Outside the stable range, $\theta_n^{p,q}(\sigma')$ is nonzero if and only
if one of the following situation holds:
\begin{enumerate}[(A)]
\item We have $p = q+t \leq 2n$. If $n\leq p-1$, then $\theta_n^{p,p}(\sigma') =
  \theta_{p-1-n}^{p,p}(\sigma')$ which is in the stable range. By
  \eqref{eq7}, if $\theta_{p-1-n}^{p,q}(L(\mu')) \neq 0$, then
  $\theta_n^{p,q}(L(\mu')) = \theta_{p-1-n}^{p,q}(L(\mu'))$.
  If $n\geq p$, then $\theta^{p,p}_n(\sigma')$ is finite dimensional
  and its associated variety is the zero orbit.

\item We have $n \leq p \leq 2n-1$ and $q + t = p + 2$. Let $\delta$
  denote the one dimensional character of $\rO(p,p+2)$ which is
  $\det_{\rO(p)}$ on $\rO(p)$ and trivial on $\rO(p+2)$.  Then $\delta
  \theta_n^{p,p+2}(\sigma') = \theta_{p-n}^{p,p+2}(\sigma')$ and we
  are back to the stable range.  By \eqref{eq7}, $\delta
  \theta_n^{p,q}(L(\mu')) = \theta_{p-n}^{p,q}(L(\mu'))$.
\end{enumerate}
Finally in Case (A), it is possible that $\theta_n^{p,q}(L(\mu')) \neq
0$ but $\theta_{p-1-n}^{p,q}(L(\mu')) = 0$. Most of these lifts
$\theta_n^{p,q}(L(\mu'))$'s are non-unitarizable. This situation
arises because the maximal Howe quotient~$\Theta(\sigma')$ is
reducible. It is possible to analyze of the $K$-types of
$\Theta(\sigma')$ as in~\cite{Lo} and compute its associated
cycles. This is tedious so we will omit this case.

\section{Associated cycles of theta lifts of unitary lowest weight
  modules} \label{S5}

\subsection{} In this section, we assume the notation in \Cref{S34}:
$G' = \spl(2n,\bR)$, $H = G^{p, q+t}$ contains $G \times K^t =
G^{p,q} \times G^{0,t}$ and $p,q,t$ satisfies \eqref{eq16}.  We pick a
$\mu \in \cR(\wtK^t,\sY_2)$ and we let $L(\mu') = \theta(\mu)$ be the
lowest weight $(\fgg', \wtK')$-module.  By \Cref{P24}
$\theta^{p,q}(L(\mu'))$ is the full theta lift.

In Lemma \ref{L33}, we define the grade module $\Gr
\theta^{p,q}(L(\mu'))$ via a natural filtration $V_j = \cU_j(\frakg)
V_0$ on $\theta^{p,q}(L(\sigma'))$ where $V_0$ is the lowest degree
$\wtK$-type. Similarly we define the graded module $\Gr \theta^{p,
  q+t}(\sigma')$ via a natural filtration $E_j = \cU_j(\frakh) E_0$ on
$\theta^{p, q+t}(\sigma')$ where~$E_0$ is the lowest degree
$\wtK_H$-type.  The following lemma is a commutative version of
\Cref{P24}.

\begin{lemma}\label{lem:res.c}
As $(\cS(\fpp), \wtK^{p,q})$-modules, 
\[
\Gr \theta^{p,q}(L(\mu')) \cong (\Gr\theta^{p, q+t}(\sigma')\otimes
\mu)^{K^t}.
\]
\end{lemma}

The proof is given in  \Cref{sec:proof.res}.

\subsection{The moment maps}\label{S53}
We will describe some moment maps. These maps are given explicitly in
terms of complex matrices in \Cref{sec:B}.

With reference to \eqref{eq2}, we denote the following moment maps
with respect to the dual pairs:
\[
\begin{array}{l|l|l}
  (G, G') & (G^{p,0}, G') & (G^{0,q},G') 
  \\ \hline 
  \psi \colon W \to \fpp'^* & \psi^+ \colon W^+ \to (\fpp'^+)^* &
  \psi^- \colon W^- \to (\fpp'^-)^* 
\end{array}
\]
We have the decomposition $W = W^+\oplus W^-$, $\fpp' =
\fpp'^+\oplus \fpp'^-$ and $\psi = \psi^+ \oplus \psi^-$.

The containment $H = G^{p, q+t} \supset G \times K^t$ gives a
decomposition $W_H = W \oplus W_2 = (W^+ \oplus W^-) \oplus W_2$.
If we replace $G$ by $H$ in the above table, then we have the moment
maps $\psi_H = \psi^+\oplus \psi_H^-$ where $\psi_H^- \colon W_H^- =
W^-\oplus W_2\to (\fpp'^-)^*$ is the moment map for pair $(G^{0,q+t},G')$. 
With respect to dual pair $(K^t, G')$, we have 
\[\psi_2 : W_2  \longrightarrow
(\fpp'^-)^* = 0\oplus (\fpp'^-)^* \subset \fpp'^*.\] Finally we get
$\psi_H^- = \psi^- + \psi_2 \colon W_H^- = W^- \oplus W_2 \to
(\fpp'^-)^*$.

\medskip

On the other side of \eqref{eq2}, we have $\phi_H \colon W_H \to
\fpp_H^*$ and $\phi \colon W \to
\fpp^*$. Let $\pr\colon W\oplus W_2\to W$ be the natural projection
and $\pr_H : \fpp_H^* \to \fpp^*$ be the projection induced from
 $\fpp\hookrightarrow \fpp_H$.  They form a commutative
diagram:
\begin{equation} \label{eq19}
\xymatrix{
\makebox[0pc][r]{$W_H =$ } W \oplus W_2 \ar[r]^{\ \ \ \ \pr} \ar[d]_{\phi_H} 
& W \ar[d]^{\phi}
\\
\fpp_H^* \ar[r]_{\pr_H} & \fpp^*.}
\end{equation}

For the rest of this paper, we refer to the orbits in
  \Cref{sec:lift.char} and we set
\begin{inparaenum}[(i)]
\item 
$\cO_H = \cO_{p,q+t} =\theta(0,G',H)$ in $\fpp_H^*$,
\item  $\cO' =\cO_d'$ in $(\fpp'^-)^*$,
\item $\cO = \cO_{p,q,t}=\theta(\cO',G',G)$ in $\fpp^*$ and
\item $\bcN = \bcNp \times \bcNn = (\psi^+)^{-1}(0)
  \times (\psi_H^-)^{-1}(0)$ to be the null cone corresponding to the
  dual pair $(H,G')$.
\end{inparaenum}

\begin{lemma} \label{lem:orbit.O1} We have $\pr_H(\bcO_H) =
  \theta(\bcOp; G', G)$ which is the Zariski closure of the
  $K_\bC$-orbit $\cO$.

\end{lemma}
\begin{proof}  We have $\overline{\cO_H} = \theta(0; G',
  H)=\phi_H(\bcN) $.  By~\eqref{eq19}, $\pr_H(\overline{\cO_H}) = \pr_H\circ \phi_H(\bcN)
  = \phi \circ \pr(\bcN)$ so it suffices to show that $\pr(\bcN) =
  \psi^{-1}(\bcOp)$.
  
  Let $(w,w_2) \in W^- \oplus W_2$. Then $(w,w_2) \in \bcNn$ if and
  only if $\psi^-(w) + \psi_2(w_2) = 0$. Hence $(w^+,w) \in
  \pr(\bcN)$ if and only if $w \in (\psi^-)^{-1}(\psi_2(W_2))$ and
  $w^+ \in \bcNp$.  Since $\bcOp = \psi_2 (W_2)$, we have
  \[
  \pr(\bcN) = \psi^+(0) \oplus (\psi^-)^{-1}(\psi_2(W_2)) =
  \psi^{-1}(\psi_2(W_2) )= \psi^{-1}(\bcOp)
  \] 
  as required. This proves the lemma.
\end{proof}

\subsection{} 
Let $\varsigma_H$ denote the lowest $\wtU(W_H)$-type of $\sY_H$. We
recall $B_H = \varsigma_H^*|_{\wtK_H} \otimes \Gr(\theta^{p,
  q+t}(\sigma'))$. Let $B = \varsigma_H^*|_{\wtK} \otimes
\Gr(\theta^{p,q}(L(\mu')))$.  We set $\tau :=
\varsigma_H|_{\wtK^t}\otimes \mu$ as in \eqref{eq:def.tau}. Then by
\Cref{lem:res.c}, we have
\begin{equation} \label{eq17} 
B = (B_H \otimes \tau)^{K^t}.
\end{equation}
Since $B_H$ is a $(\bC[\overline{\calO_H}],K_H)$-module, applying
\Cref{lem:orbit.O1} to \eqref{eq17} shows that $B$ is a
$(\bC[\bcO],K_\bC)$-module.  Let $\sB$ be the quasi-coherent sheaf on
$\bcO$ associated with $B$.  In particular, $\Supp(\sB) \subseteq
\bcO$.

Fix $x \in \cO$ and let $i_x \colon \set{x} \to \bcO$ be the inclusion
map. Let $\cI_{x}$ be the maximal ideal in $\cS(\fpp)$ defining $x$.
Let $q_H := \pr_H \circ \phi_H : W_H \rightarrow \fpp^*$ and let $Z =
q_H^{-1}(x) \cap \bcN$ be the set theoretic fiber.

\begin{lemma} \label{L55} We have $\bC[Z] =
  \bC[\bcN]/\cI_x\bC[\bcN] = \bC[x] \otimes_{\bC[\bcO]}
  \bC[\bcN]$.
\end{lemma}

\begin{proof}
  Let $\bcN_{x}$ be the scheme theoretic fiber of $x$. We claim that
  $\bcN_{x}$ is reduced and thus equals to~$Z$. Indeed in
  characteristic zero, a generic scheme theoretic fiber is
  reduced. Since $x$ generates the dense open orbit $\cO$ in $\bcO$,
  $\bcN_{x}$ is reduced and the claim follows. Taking regular
    functions of $Z = \bcN_{x}$ gives the lemma.
\end{proof}

By \Cref{prop:iso.cN}, $B_H = (\bC[\bcN]\otimes A)^{K'_\bC}$ where $A =
\varsigma_H|_{\wtK'} \otimes \sigma'^*$.  Since $\cS(\fpp)$ is $K'_\bC
\times K_{\bC}^t$-invariant, \eqref{eq17} and \Cref{L55} give
\begin{equation}\label{eq:irep.1}
\begin{split}
  i_{x}^*\sB =& B/\cI_{x}B = (B_H/\cI_{x}B_H \otimes \tau)^{K_{\bC}^t} 
  = ( (\bC[\bcN]/\cI_{x}\bC[\bcN] \otimes A)^{K'_\bC}\otimes
  \tau)^{K_{\bC}^t} \\
=&   (\bC[Z] \otimes A\otimes \tau)^{K'_\bC\times K_{\bC}^t}.
\end{split}
\end{equation}

\medskip

In \Cref{sec:B}, we see that $\psi^- : W^- \rightarrow (\fpp'^-)^*$ is
surjective if and only if $q \geq n$.  From now on we split out
calculations into two cases, depending on whether $\psi^-$ is
surjective or not surjective.



\subsection{Case I} \label{S57}
We assume that $\psi^- : W^- \rightarrow (\fpp'^-)^*$ is surjective,
i.e. $q \geq n$. 

Let $Y = \pr(Z) = \phi^{-1}(x)\cap \pr(\bcN)$.  Fix $y = (y^+,y^-)\in
Y\subset W^+\oplus W^-$. Let $Z_y = \pr^{-1}(y)\cap Z$.  Let $x' =
-\psi^-(y^-) \in (\fpp'^-)^*$ and let $Z_{x'} = (\psi_2)^{-1}(x')$ in
$W_2$. We consider the diagram:
\begin{equation}\label{eq18}
 \xymatrix{
  Z_{x'}&  \ar[l]_T^{\cong} Z_y\ar[d]^{\pr}\ar@{^(->}[r] & Z\ar@{->>}[d]^{\pr}\\
  & \Set{ y } \ar@{^(->}[r]^{i_y} &  Y.
}
\end{equation}
The map $T$ will be given in \Cref{lem:geo.I}(iii) below.  Let $S_y
= \Stab_{K_\bC \times K'_\bC}(y)$ and $K_{x'}' = \Stab_{K'_\bC}(x')$.
We now state our key geometric \Cref{lem:geo.I}. Its proof is given in
\Cref{SB7}.


\begin{lemma}\label{lem:geo.I}
Suppose $\psi^- : W^- \rightarrow (\fpp'^-)^*$ is a surjection.
\begin{enumerate}[(i)]
\item Then $Y$ is a single $K_{x}\times K'_\bC$-orbit
  generated by $y$ so $Y \simeq (K_{x}\times K'_\bC)/S_y$.

\item There is group homomorphism $\beta \colon K_{x} \to K_{x'}'$ 
  such that 
\[
S_y = \Set{(k,\beta(k))| k \in K_{x}}.
\]
We denote the right hand side by $K_{x} \times_\beta K_{x'}'$.

\item There is a bijection $T : Z_y \rightarrow Z_{x'}$ such that $T$
  commutes with the actions of $K_\bC^t$ and
\begin{equation} \label{eq22}
T((k,\beta(k)) z) = \beta(k) z 
\end{equation}
for all $z \in Z_y$ and $(k,\beta(k)) \in S_y = K_{x} \times_\beta
K_{x'}'$.
\end{enumerate}
\end{lemma}

\begin{proof}[Proof of \Cref{thm:LL}]
Let $\sO_Z$ denote the structure sheaf of $Z$. Clearly $\bC[Z] =
((\pr|_{Z})_* \sO_Z)(Y)$. By \Cref{lem:geo.I}(i) above, $Y$ is a
single $K_{x}\times K'_{\bC}$-orbit. Let $\cI_y$ be the ideal of $y$
in $\bC[Y]$. We recall that $Y$ is affine. Again by the generic reduceness
of scheme theoretical fiber in characteristic $0$, $\bC[Z_y] =
\bC[Z]/\cI_y\bC[Z]$.  Therefore by \cite{CPS},
\begin{equation}\label{eq:iso.Z.1}
\bC[Z] = \Ind_{S_y}^{K_{x}\times K'_\bC}  (i_y^* (\pr|_{Z})_*\sO_Z)
 = \Ind_{S_y}^{K_{x}\times K'_\bC} (\bC[Z]/\cI_y \bC[Z])
 = \Ind_{S_y}^{K_{x}\times K'_\bC} \bC[Z_y].
\end{equation}
Putting \eqref{eq:iso.Z.1} into \eqref{eq:irep.1}, we have
\begin{eqnarray*} 
  i_{x}^*\sB & = & \left(\bC[Z] \otimes A \otimes \tau
  \right)^{K_\bC' \times K_\bC^t} = \left(\Ind_{S_y}^{K_x \times K'_\bC} \bC[Z_y]
    \otimes A \otimes \tau \right)^{K'_\bC\times K_\bC^t}  \\
& = & \left(\Ind_{S_y}^{K_x \times K'_\bC}(\bC[Z_{y}] \otimes
    \tau)^{K_\bC^t} \otimes A \right)^{K'_\bC}.
\end{eqnarray*}
By \Cref{lem:geo.I}(iii), $T$ induces an isomorphism $(\bC[Z_{y}]
\otimes \tau)^{K^t} = (\bC[Z_{x'}] \otimes \tau)^{K^t}$ of
$S_y$-modules through $\beta$.

By \eqref{eq17c}, $(\bC[Z_{x'}] \otimes \tau)^{K^t} \cong \chi_{x'}$.
Moreover, $A = \varsigma_H|_{\wtK'} \otimes \sigma'^* \cong
\varsigma|_{\wtK'}\otimes \varsigma_2|_{\wtK'}$
Hence 
\begin{equation}\label{eq:chix.I}
\begin{split}
  \chi_x = i_x^*\sB =& \left(\Ind_{S_y}^{K_x\times K'_\bC} (\chi_{x'}
    \otimes A) \right)^{K'_\bC} \cong (\chi_{x'}\otimes A) \circ \beta
  = (\tchi_{x'}\otimes \varsigma|_{\wtK'})\circ \beta
\end{split}
\end{equation}
as a representation of $K_{x}$.  The isotropy representation of
$\Gr\theta(L(\mu'))$ at $x$ is
\[
\tchi_{x} = \varsigma|_{\wtK}\otimes i_x^*\sB =
\varsigma|_{\wtK}\otimes (\tchi_{x'}\otimes \varsigma|_{\wtK'})\circ
\beta.
\]
This proves (iii). 

Suppose that $L(\mu')\neq 0$.  By \Cref{T52}, $\tchi_{x'}\neq 0$ so
$\tchi_{x}\neq 0$.  This implies that $\theta^{p,q}(L(\mu')) \neq 0$
and proves (i). We have seen before that $\Supp(\sB) \subseteq
\bcO$. Since $\chi_{x}\neq 0$, $x \in \Supp(\sB)\subset \bcO$ and
$K_\bC \cdot x = \cO$. Hence $\AV(\theta^{p,q}(L(\mu')) = \bcO$. This
proves (ii). Part~(iv) is an immediate consequence of (ii) and
(iii). This completes the proof of \Cref{thm:LL}.
\end{proof}

\subsection{Case II} \label{S58} We assume that $\psi^- : W^-
\rightarrow (\fpp'^-)^*$ is not surjective, i.e. $t> n > q$.  Then
$K_{x}$ has a Levi decomposition $K_{x} = L\ltimes N$ (see p. 184
in~\cite{Hu}) where $L$ is the Levi part and $N$ is the unipotent
part.  By the calculation in \Cref{SB12}, we have $L \cong \diag
K_\bC^q\times K_\bC^{p-2q}$.

We would like to mimic \eqref{eq18} in Case I which constructs an
orbit $Y$. More precisely we will construct in \Cref{SB12} an (affine)
algebraic set $\cM$ and a surjective algebraic morphism $\pi : Z
\rightarrow \cM$ with the following properties:
\begin{enumerate}[(A)]
\item \label{item:ind.A} Let $Q = L \times K_\bC^t \times K_\bC'
  \simeq \diag K_\bC^q\times K_\bC^{p-2q} \times K_\bC^t \times
  K_\bC'$. There is a $Q$-action on $\cM$ such that $\pi$ is $Q$-equivariant.

\item The set $\cM$ is an $Q$-orbit. Fix $m\in \cM$.  We form
  the following  set theoretical fiber:
\[
\xymatrix{
\bcN_s&  \ar[l]^{\cong}_T Z_m \ar@{^(->}[r] \ar[d] & Z \ar[d]^{\pi} \\
&\Set{ m } \ar@{^(->}[r] & \cM
}
\]
By the same argument as in the proof of \Cref{L55}, $Z_m$ is equal to
the scheme theoretical fiber $Z \times_{\cM} \Set{m}$ because the
latter is reduced.

\item Let $Q_m$ be the stabilizer of $m$ in $Q$.
We set $K_s = K^{p-2q} \times K^{t-q}$ and
   $K_s'\cong \rU(n-q)$.
 Then 
\[
\begin{split}
  Q_m =& \diag' K_\bC^q \times K_\bC^{p-2q} \times K_\bC^{t-q} \times
  K'^{n-q}_\bC \\
  \cong & \triangle' K_{\bC}^q \times K_{s \bC} \times K_{s \bC}'
  \cong L \times_{\gamma} K_{s \bC} \times K'_{s \bC}
\end{split}
\]
Here $\diag' K_\bC^q$ embeds diagonally into $\diag K_\bC^q \times
K_\bC^t \times K'_\bC$, $K_\bC^{t-q}$ lies in $K_\bC^t$ and
\[
\gamma\colon L = \diag K_{\bC}^q \times K_\bC^{p-2q} \longrightarrow K_\bC^{p-2q}
\subset {K_s}_\bC
\]
 is the natural projection.

\item \label{item:ind.D} Let $(G_s, G_s') =
  (\rO(p-2q,t-q),\Sp(2n-2q))$ be the dual pair with maximal compact
  subgroup $(K_s,K_s')$ as in (C). Then $(G_s,G_s')$ is in the stable
  range (c.f. \eqref{eq:range.C}). We put subscript $s$ for all
  objects with respect to this pair. For example, $\sigma'_s$ denotes
  the genuine character of $\wtG_s'$.  Correspondingly we have the
  closed null cone $\bcNs =\psi_s^{-1}(0)$.

\item Let $Q_m$ act on $\bcNs$ with trivial $\diag' K_\bC^q$-action.
  Then there is a $Q_m$-equivariant bijection $T : Z_m \rightarrow
  \bcNs$.
  \item \label{item:ind.Z2} By (B), we have $\cM \simeq Q/Q_m$ and
\begin{equation}\label{eq:iso.Z.2}
  \bC[Z] = \Ind_{Q_m}^Q \bC[Z_m]  = \Ind_{Q_m}^Q \bC[\bcNs].
\end{equation}

\end{enumerate}
The proofs of (A) to (E) are given in \Cref{SB12}. Part
\eqref{item:ind.Z2} follows from~\cite{CPS} using a similar argument
as \eqref{eq:iso.Z.1}.

\begin{proof}[Proof of \Cref{TC}]
Putting \eqref{eq:iso.Z.2} into \eqref{eq:irep.1}, the isotropy
representation of $B$ at $x$ is
\begin{equation}\label{eq:chix.II}
\begin{split}
\chi_x|_{L} = i_{x}^*\sB|_{L} =& \left(\Ind_{Q_m}^{L\times  K'_\bC\times
    K_\bC^t} \bC[\bcNs]\otimes A \otimes \tau \right)^{K'_\bC\times K^t_\bC}\\
=& \left( (\bC[\bcNs] \otimes A|_{K'_s})^{K'_s} \otimes \tau \right)^{K_\bC^{t-q}}.
\end{split}
\end{equation}

Note that $A|_{K'_s}\cong \varsigma_s\otimes \sigma'^*_s$,
$\varsigma|_{\wtK_s} = \varsigma_s$, $\mu|_{\wtK^{t-q}\times \wtK^{q}}
= \varsigma^*|_{\wtK^{t-q}\times \wtK^{q}}\otimes
\tau|_{K_\bC^{t-q}\times K^{q}}$. By \Cref{cor:iso.c}, $ (\bC[\bcNs]
\otimes A|_{K'_s})^{K'_s} = B_s$.  Finally, we get
\[
\begin{split}
  \tchi_x|_{\wtL} =& \varsigma|_{\wtL}\otimes
  (\varsigma_s^*|_{\wtK^{p-2q}\times\wtK^{t-q}}\otimes\theta^{p-2q,t-q}(\sigma'_s)\otimes
  \tau|_{K^{t-q}\times K^{q}})^{K^{t-q}}\\
  =& (\theta^{p-2q,t-q}(\sigma'_s) \otimes \mu|_{\wtK^{t-q}\times
    \wtK^{q}})^{K^{t-q}} \otimes \varsigma^2|_{K^{q}}.
\end{split}
\]
as a representation of $\wtK^{q}\times \wtK^{p-2q} \twoheadrightarrow
\wtL$. This proves (ii).

By \eqref{eq:chix.II}, $\theta^{p,q}(L(\mu')) = 0$ implies that
$\tchi_x = 0$. Conversely if $\theta^{p,q}(L(\mu'))$ is non-zero then
by \Cref{T61} below (whose proof does not depend on the result of this
subsection) $\tchi_x$ is non-zero.  This proves (iii).

The proof of (iv) is the same as that of \Cref{thm:LL}~(ii)~and~(iv).  We have
$\Supp(\sB) \subseteq \bcO$. If $\theta^{p,q}(L(\mu')) \neq 0$, then
$\tchi_{x}\neq 0$ so $x \in \Supp(\sB)\subset \bcO$ and $K_\bC \cdot x
= \cO$. Hence $\AV(\theta^{p,q}(L(\mu')) = \bcO$. This proves (iv).
This completes the proof of \Cref{TC}.
\end{proof}

\subsection{Proof of \Cref{thm:KS}} \label{sec:kspec} In this section,
we assume the notation of the Sections~\ref{S57} and \ref{S58}: $x \in
\cO = \cO_{p,q,t}=\theta(\cO',G',G)$ and $B = \varsigma^*|_{\wtK}
\otimes \Gr(\theta(L(\mu'))$. Since $\theta^{p,q}(L(\mu')) \cong
\varsigma|_{\wtK} \otimes B$, \Cref{thm:KS} follows from Proposition
\ref{T61}(ii) below.  We emphasize that the proof of the proposition
is independent of the calculations of $\chi_x$ in \Cref{S57} and
\Cref{S58}.

\begin{prop} \label{T61} Suppose that $p,q,t$ satisfies \eqref{eq16}
  and $n> \min\set{q,t}$. Let $\chi_x$ be the $K_x$-module calculated
  in \eqref{eq:chix.I} and \eqref{eq:chix.II}. 
\begin{enumerate}[(i)]
\item We have $\sB = (i_\cO)_* \sL$ as $(\sO_{\bcO},K_\bC)$-modules
where $i_\cO\colon
\cO\hookrightarrow \bcO$ is the natural open embedding and $\sL$ is
the $K_\bC$-equivariant coherent sheaf with fiber $\chi_x$ at $x$ in
sense of~\cite{CPS}. 

\item  As $K_\bC$-modules,
\[
B = \Ind_{K_x}^{K_\bC} \chi_x.
\]
In particular, $B \neq 0$ if and only if $\chi_x \neq 0$.
\end{enumerate}
\end{prop}

\begin{proof}
  By definition $\sB(\bcO) = B = \varsigma^*|_{\wtK} \otimes
  \theta(L(\mu'))$. On the other hand, $i_\cO^* \sB\cong \sL$ and
  $\sB(\cO) = (i_\cO)_* \sL(\bcO)= \Ind_{K_x}^{K_\bC}\chi_x$ by
  \cite{CPS}. We will show (ii), i.e. $\sB(\bcO) = \sB(\cO)$ under the
  restriction map. Then (i) follows because both $\sB$ and $(i_\cO)_*
  \sL$ are quasi-coherent sheaves over an affine scheme with the same
  space of global sections.

  By \Cref{lem:res.c} and \Cref{prop:iso.cN}
\[
\sB(\bcO) = (\bC[\bcN]\otimes A \otimes \tau)^{K'_\bC\times K_\bC^t} =
(\bC[\cN]\otimes A \otimes \tau)^{K'_\bC\times K_\bC^t}.
\]

On the other hand, since $\fpp$ is $K'_\bC\times K_\bC^t$-invariant,
localization commutes with taking $K'_\bC\times K_\bC^t$-invariants.
Let $\cD = (\phi \circ \pr)^{-1}(\cO)\cap \cN$ and we consider
\[
\xymatrix{
\cD\ar@{^(->}[r]^{i_\cD}\ar[d]_{\phi\circ \pr|_\cD}& \cN\ar[d]^{\phi \circ \pr}\\
\cO \ar@{^(->}[r]^{i_{\cO}}& \bcO.
}
\]
Since $i_{\cO}$ is an open embedding, it is flat and we have
(c.f. Corollary 9.4 in \cite{Ha})
\[
\Gamma(\cO, i_{\cO}^*(\phi \circ \pr)_*\sO_\cN) = 
\Gamma(\cO, (\phi\circ \pr|_{\cD})_*i_\cD^*\sO_\cN) = 
\Gamma(\cD,\sO_{\cD}) = \bC[\cD].
\]
This gives $\sB(\cO) = (\bC[\cD]\otimes A \otimes
\tau)^{K'_\bC\times K_\bC^t}$.  Therefore, it suffices to show that
\begin{equation}\label{eq:D.iso}
  H^0(\cN,\sO_\cN) = \bC[\cN] \to \bC[\cD] = H^0(\cD,\sO_\cN)
\end{equation}
is an isomorphism.  

\begin{lemma} \label{L62} Suppose $n > \min \Set{ q,t }$, then
  $\cD$ is a $K_\bC \times K_\bC' \times K_\bC^t$-orbit and
  $\partial \cD =
  \cN - \cD$ has codimension at least $2$ in~$\cN$.
\end{lemma}

The proof of the lemma is given in \Cref{SB5}.

We continue with the proof of the theorem.  We note that $\cN$ is a
$(K_H)_\bC \times K_\bC'$-orbit and $\cD$ is a $K_\bC \times K_\bC^t
\times K_\bC'$-orbit. Now \eqref{eq:D.iso} follows from Theorem 4.4 in
\cite{CPS}.  This proves \Cref{T61}.
\end{proof}

\section{Special unipotent representations} \label{S6}

We will briefly review special unipotent primitive ideals and
representations in Chapter~12 in \cite{V2}. Also see Section 2.2 in
\cite{T}. We will apply these to $\theta^{p,q}(L(\mu'))$. 
  
Let $\fgg = \frakso(p+q,\bC)$. We consider $\theta^{p,q}(L(\mu'))$
where $q > n > t$, $q + t \geq 2n$, $\dim \mu = 1$ and $L(\mu') =
\theta(\mu)$.  The infinitesimal character of $\theta^{p,q}(L(\mu'))$
corresponds to the weight $\lambda = (\delta_{t}, \delta_{p+q-2n},
\delta_{2n-t+2})$ under the Harish-Chandra parametrization
\cite{Lo}. Here $\delta_N = (\frac{N}{2}-1, \frac{N}{2}-2, \ldots)$
denotes the half sum of the positive roots of $\frakso(N)$, and we
insert or remove zeros from $\lambda$ if the string of numbers is too
short or too long. The restriction of $\theta(L(\mu'))$ to
$(\fgg,(\wtK^{p,q})^0)$ decomposes into a finite number of irreducible
$(\fgg,(\wtK^{p,q})^0)$-submodules.  Let~$\theta(L(\mu'))^0$ denote
one of these irreducible $(\fgg,(\wtK^{p,q})^0)$-submodules. Since
$\cO = \cO_{p,q,t}$ is also a $(\wtK^{p,q})^0$-orbit,
$\theta(L(\mu'))^0$ also have associated variety $\cO$.

We claim that the weight $\lambda$ also represents the infinitesimal
character of $\theta(L(\mu'))^0$ as a
$(\fgg,(\wtK^{p,q})^0)$-module. Indeed we may have an ambiguity only
if $p+q$ is even where the infinitesimal character is either $\lambda$
or $s(\lambda)$. Here $s$ is the involution induced by the outer
automorphism of $\fgg$. In this case $t$ is even and the weight
$\lambda$ contains a zero in the string of numbers. Hence $\lambda$ and
$s(\lambda)$ represent the same infinitesimal character and proves our
claim.

Under the Kostant-Sekiguchi correspondence, $\calO$ generates a
nilpotent $\SO(p+q,\bC)$-orbit $\calC$ in $\fgg^*$. The orbit
$\calC$ has the same Young diagram as that of $\calO$ in
\eqref{eqP16} less the plus and minus signs. Let $J$ denote the
primitive ideal of $\theta(L(\mu'))^0$ in $U(\fgg)$. The ideal~$J$
has a filtration $\{ J_s = U_s(\fgg) \cap J : s \in \bN
\}$. By \cite{VoAss}, $\Gr(J)$ cuts out the variety
$\overline{\calC}$ in~$\fgg^*$.
  
Let $\Phi$ and $R$ be the roots and the root lattice of $\fgg$.  Let
$\fgg^\vee$ denote the simple Lie algebra with $\Phi$ as coroots. In
particular $\fgg^\vee = \frakso(p+q,\bC)$ if $p+q$ is even and
$\fgg^\vee = \fraksp(p+q-1,\bC)$ if $p+q$ is odd. We refer to
\cite{T2} on the order reversing map $d$ (resp. $d^\vee$)
from the set of complex nilpotent orbits in $\fgg^*$
(resp. $(\fgg^\vee)^*$) to the complex nilpotent orbits in
$(\fgg^\vee)^*$ (resp. $\fgg^*$).  The orbit $\calC$ is called special
if it is in the image of $d^\vee$ and for a special orbit $\calC$, we
have $d^\vee(d(\calC)) = \calC$.
  
Suppose $\calC$ is a special orbit. By the Jacobson-Morozov theorem,
let $\{ X^\vee, H^\vee, Y^\vee \}$ be the $\fraksl_2$-triple such that
$X^\vee \in d(\calC)$. We may assume that $\frac{1}{2} H^\vee$ lies in
$\bC \otimes_\bZ R$. In this way~$\frac{1}{2} H^\vee$ defines an
infinitesimal character of $U(\fgg)$ via the Harish-Chandra
homomorphism. If $d(\calC)$ has Young diagram $(a_1, a_2, \ldots,
a_s)$ then $\frac{1}{2} H^\vee = (\delta_{a_1+1}, \delta_{a_2 + 1},
\ldots, \delta_{a_s + 1})$.

Suppose $\frac{1}{2} H^\vee$ gives the same infinitesimal character
$\lambda$ as that of $\theta(L(\mu'))^0$. Let $J_\lambda$ denote the unique
maximal primitive ideal of $U(\fgg)$ with infinitesimal character
$\frac{1}{2} H^\vee$ \cite{Du}. We will call~$J_\lambda$ a special unipotent primitive
ideal. By Corollary A3 in \cite{BV}, the variety cut out by $\Gr(J_\lambda)$
in $\fgg^*$ is the same as that of $\Gr(J)$, namely
$\overline{\calC}$. By Corollary 4.7 in \cite{BK}, $J_\lambda = J$. We say
that~$\theta(L(\mu'))^0$ is a special unipotent representation.

\begin{prop} \label{P86} Suppose $q + t \geq 2n$, $q \geq n \geq t$ and
  $\dim \mu = 1$. Then $\calC$ is a special orbit and
  $\theta(L(\mu'))^0$ is a special unipotent representation.
\end{prop}

\begin{proof}
  The orbit is special could be read off from page 100 in \cite{CM}.
  Indeed $d(\calC) = (p+q-2n-1,2n-t+1,t-1,\epsilon)$ where $\epsilon =
  0$ if $t$ is odd (i.e. $\fgg^\vee$ of type ${\mathrm{C}}$) and
  $\epsilon = 1$ if $t$ is even (i.e. $\fgg^\vee$ of type ${\mathrm{D}}$).
  Furthermore $d^\vee d(\calC) = \calC$. We have $\frac{1}{2} H^\vee =
  (\delta_{t}, \delta_{p+q-2n}, \delta_{2n-t+2}) = \lambda$ which is
  the infinitesimal character of $\theta(L(\mu'))^0$. The conclusion
  that $\theta(L(\mu'))^0$ is special unipotent follows from the
  discussion prior to the proposition.
\end{proof}

The above argument applies to $\theta(\sigma')$ too by setting $t =
0$. It is easier and we leave the details to the reader.

\section{Other dual pairs} \label{S7} Most of our methods and results
extend to the two dual pairs $(G^{p,q},G')$ in Table \ref{tab:ls}
below. We have omitted them in the main body of this paper in order to
keep this paper simple. In this section, we will briefly describe
these two dual pairs. 

\begin{table}[h,t]
  \centering \small \setlength{\tabcolsep}{2pt}
  \begin{tabular}{cc|c|c|c|c}
     $G^{p,q}$ & $G'$  & Stable range & Case I & Case II& $\codim \, \partial
    \cD \geq 2$\\
    \hline
 $\rU(p,q)$ & $\rU(n_1,n_2)$& $p,q \geq n_1+n_2$ & $q \geq
    n_1,n_2$ & $\max \Set{n_1,n_2} > q$ & $\begin{array}{l} \max\Set{n_1,n_2} > 
      \\ \min\Set{t,n_1,n_2} \end{array}$
    \\
    \hline
$\Sp(2p,2q)$ & $\rO^*(2n)$ & $p,q\geq n$ &  
    $2q \geq n$ &  $2q < n $ &
    $\begin{array}{c}n > 2t \\
      \text{or $n$ is odd}
    \end{array}
    $
  \end{tabular}
  \caption{List of dual pairs}
  \label{tab:ls}
\end{table}

There is also a notion of theta lifts of $K_\bC'$-orbits on
$\fpp'^*$ to $K_\bC$-orbits on $\fpp^*$, and conversely.

First we suppose $(G^{p,q},G')$ is in stable range where $G'$ is the
smaller member.  This condition is given in the second column of
\Cref{tab:ls}.  Let $\sigma'$ be a genuine unitary character
of~$\wtG'$. For the dual pair $(\Sp(2p,2q), \rO^*(2n))$, $\rO^*(2n)$
alway splits and $\sigma'$ is unique.  The local theta lift
$\theta^{p,q}(\sigma')$ is nonzero and unitarizable, and it also the
full theta lift (c.f. Proposition \ref{P32}). By an almost identical
proof as that of \Cref{thm:ACchar}, one shows that $\AC(\theta(
\sigma')) = 1 [\overline{\cO_{p,q}}]$ where $\cO_{p,q} =
  \theta(0;G',G^{p,q})$. Furthermore we have $\theta(\sigma')|_{\wtK}
= \Ind_{\wtK_x}^{\wtK_\bC} \tchi_x$ where $\tchi_x$ is the isotropy
representation \cite{Y}.

Next we consider the dual pair $(G^{t,0}, G')$. Let $\mu$ be an
irreducible genuine representation of $\wtG^{t,0}$ such that $L(\mu')
= \theta(\mu)$ is nonzero. Then $L(\mu')$ is a unitary lowest weight
Harish-Chandra module of~$\wtG'$. By \cite{Ya}, $\AV(L(\mu')^*)$ is
the Zariski closure of an orbit $\cO'$ in~$\fpp'^-$. The isotropy
representation $\tchi_{x'}$ of $L(\mu')^*$ at a closed point $x' \in
\cO'$ is also computed explicitly in \cite{Ya}.


Let $\sigma'$ be a genuine unitary character of $\wtG'$ for the dual
pair $(G^{p,q+t},G')$ in stable range.  The local theta lift
$\vartheta := \theta^{p,q}(\sigma'^* L(\mu'))$ to $\wtG^{p,q}$ is also
the full theta lift (c.f. Proposition~\ref{P24}). Again we have to
divide into Cases I and II as in Theorems \ref{thm:LL}
and~\ref{TC}. The conditions for Cases I and II are given in the
third and fourth columns of \Cref{tab:ls} respectively.

In Case I, we have results similar to that of \Cref{thm:LL}. More
precisely, $\vartheta = \theta^{p,q}(\sigma'^* L(\mu'))$ is
nonzero. Its associated variety is
\begin{equation} \label{eqn13}
\AV(\vartheta) = \theta(\overline{\cO'}; G', G) =
\theta(\AV(L(\mu')^*))
\end{equation}
and it is the Zariksi closure of single $K_\bC'$-orbit $\cO$. We fix
a closed point $x \in \cO$. Let $K_x$ and $K_{x'}'$ be the
stabilizers of $x$ and $x'$ in $K_\bC$ and $K_\bC'$ respectively.
Then there is a group homomorphism $\beta\colon K_{x}\to K'_{x'}$ such that
the isotropy representation of $\Gr\vartheta$ at $x$
is
\[
\tchi_{x} = \varsigma|_{\wtK} \otimes(\varsigma|_{\wtK'}\otimes
\sigma'^* \otimes \tchi_{x'})\circ \beta. 
\]
Therefore $\vartheta$ satisfies \eqref{eqnn13}, i.e.
\begin{equation} \label{eqn14}
\AC(\vartheta) = (\dim\tchi_x) [\bcO] = (\dim\tchi_{x'})
[\overline{\theta(\cO')}] =  \theta(\AC(L(\mu')^*)).
\end{equation}
The last column lists the conditions in Case I such that
(c.f. \Cref{thm:KS})
\[
\vartheta|_{\wtK} = \Ind_{\wtK_{x}}^{\wtK_\bC} \tchi_x.
\]

For Case II, the situation is more complicated. Equation \eqref{eqn13} continues to hold but~\eqref{eqn14}
fails in general.

\appendix

\section{Natural filtrations}

\subsection{} \label{SL41} {\it Proof of \Cref{L41}.}  The map $\nu$
in \eqref{eq:9} factors through the $\wtK'$-covariant subspace
$\sY_{\tau'}$ of type $\tau'$ of $\sY$.  Let $\sY(\tau')$,
$\sY^d(\tau')$ and $\sY_j(\tau') = \oplus_{d \leq j} \sY^d(\tau')$
denote the $\tau'$-isotypic components of $\sY$, $\sY^d$ and $\sY_j$
respectively. Since $\wtK'$ has reductive action on $\sY$ and
preserves degrees, $\sY(\tau')$ maps bijectively onto the covariant
$\sY_{\tau'}$.  Moreover $\sY^d(\tau') = \sY^d \cap \sY(\tau')$ and
$\sY_j(\tau') = \sY_j \cap \sY(\tau')$. Hence
\begin{equation} \label{eq:10}
\nu(\sY_j(\tau')) = F_j.
\end{equation}
Let $\calH(\tau')$ denote the $\tau'$-isotypic component in the
harmonic subspace $\calH(\wtK')$ of $\bC[W]$ for~$\wtK'$.  By
\cite{H2}, we have $\calH(\tau') \subset \sY^{j_0}$ and by \cite{H1}
we have
\[
\sY(\tau') =\cU(\fmm^{(2,0)}) \calH(\tau').
\]
Since $\frakm^{(2,0)}$ acts by degree two polynomials, $\sY^j(\tau') =
0$ if $j \not \equiv j_0 \pmod{2}$ and $\sY_{2j+j_0}(\tau') =
\cU_j(\fmm^{(2,0)})\cH(\tau')$.  It follows from \eqref{eq:10} that
$\nu(\cU_j(\frakm^{(2,0)}) \, \calH(\tau')) = F_{2j + j_0} =
F_{2j+j_0+1}$.

 We will prove $V_j= F_{2j+j_0}$ by induction. First we have
  $V_{0} = F_{j_0}$. Suppose $V_{j} = F_{2j+j_0} =
  \nu(Y_j)$ where $Y_j = \sY_{2j + j_0}(\tau')$. Since $V_{j+1} =
  \nu(\cU_{j+1}(\fgg)\sY_{j_0})  \subseteq    \nu(\sY_{2(j+1)+j_0})=
  F_{2(j+1)+j_0}$, it suffices to show that $F_{2(j+1)+j_0} \subseteq
  V_{j+1}$.  By \eqref{eq:8}, $Y_j + \frakm^{(2,0)} Y_j = Y_j +
  \frakp Y_j$.  Hence
\begin{eqnarray*}
F_{2(j+1)+j_0} & = & \nu(\cU_{j + 1}(\frakm^{(2,0)})\calH(\tau')) \subseteq
\nu(Y_j + \frakm^{(2,0)} Y_j) =
\nu(Y_j + \frakp Y_j) = \nu(Y_j) + \frakp
\nu(Y_j) \\ 
& = & V_{j} + \frakp V_{j} = V_{j+1}.
\end{eqnarray*}
This shows that $F_{2(j+1)+j_0} = V_{j+1}$ and completes the proof of the
lemma. \hfill{\qed}


\subsection{} \label{sec:proof.res} {\it Proof of \Cref{lem:res.c}.}
By \Cref{P24} $\theta^{p,q}(L(\mu')) \cong (\theta^{p,
  q+t}(\varsigma')\otimes \mu)^{K^t}$. This defines another filtration $E_j'
:= (E_j \otimes \mu)^{K^t}$ on $\theta^{p,q}(L(\mu'))$. Let
$v_0$ be the degree of the lowest degree $\wtK$-type $V_0$ and $e_0$
be the degree of the lowest degree $\wtK_H$-type $E_0$. Let $j_0$ be
the smallest integer such that $E'_j\neq 0$.  In
order to prove \Cref{lem:res.c}, it suffices to prove that
\begin{equation} \label{eq17a}
V_j= E'_{j+j_0}.
\end{equation}
Indeed, by Lemma \ref{L33} and \eqref{eqn20}, we have a surjection 
\[
\xymatrix{
\eta_j : \displaystyle\sum_{a + b = 2j + e_0} \sY_a \otimes (\sY_2^*)_b \ar@{->>}[r]
& E_j.}
\]
Let $l_0$ be the degree of $\mu'$. Then $L_l := ((\sY_2^*)_{2l+l_0} \otimes
\mu)^{K^t} = ((\sY_2^*)_{2j+l_0+1} \otimes
\mu)^{K^t}$ is the natural filtration on
$L(\mu')^*$. 

Taking the $\mu$-coinvariant of $\eta_j$ gives a surjection
\[
\xymatrix{
\eta_j' : \displaystyle\sum_{a+2l=2j+e_0-l_0} \sY_a\otimes L_l = \displaystyle\sum_{a + b = 2j + e_0} \sY_a \otimes \left((\sY_2^*)_b \otimes
\mu\right)^{K^t} \ar@{->>}[r]& (E_j
\otimes \mu)^{K^t} = E'_j.
}
\]
Hence the image $E_j'$ of $\eta_j'$ is
also the filtration for $\theta^{p,q}(L(\mu'))$ defined before
Lemma \ref{L33} up to a degree shifting. Now \eqref{eq17a} follows from Lemma
\ref{L33}. \qed


\section{Geometry}\label{sec:B}

\subsection{Explicit moment maps} \label{sec:A.mmap} In this section, we will
denote the space of $p$ by $n$ complex matrices (resp. symmetric
matrices) by $M_{p,n}$ (resp. $\Sym^n$). Let $I_{p,n} = (a_{ij})$
denote the matrix in $M_{p,n}$ such that $a_{ij} = \delta_{ij}$. 

We identify
\[
W_H = W^+ \oplus W_H^- = W^+ \oplus W^- \oplus W_2
\]  
with the set of complex matrices
\[
M_{p+q+t,n} = M_{p,n} \oplus M_{q+t,n} = M_{p,n} \oplus M_{q,n} \oplus
M_{t,n}.
\]
We will denote an element in $W_H$ by $w = (w^+;w^-) = (w^+;w_1,w_2)$
and an element in $W = W^+ \oplus W^-$ by $(w^+;w_1)$.  
The projection
map $\pr : W_H \rightarrow W$ is given by $\pr(w^+; w_1,w_2) =
(w^+;w_1)$ and $\psi^- : W^- \rightarrow \fpp'^- = \Sym^n$ is given
by $\psi^-(w_1) = (w_1)^T  w_1 \in \Sym^n$. In particular~$\psi^-$
is surjective if and only if $q \geq n$.

As always we will denote $\rO(p,\bC)$ by $K_\bC^p$.
An element $(o_p,o_q,g)\in K^p_\bC\times K^q_\bC \times K'_\bC$ acts on $W$ by 
$(o_p,o_q,g) \cdot (w^+,w_1) = (o_pw^+g^{-1}, o_qw_1g^T)$.
Let $E_{p,n} = \left( \begin{smallmatrix} I_n\\0 \\iI_n\end{smallmatrix}\right)$ be the $p$
by $n$ matrix with $n$-linearly independent column vectors whose
column space is isotropic.

Let
$P_{p,n}\subset K_\bC^p$ be the stabilizer of the isotropic subspace
spanned by the columns of $E_{p,n}$. It is a maximal parabolic
subgroup of $K_\bC^p$. Then the column space of the complex
conjugation $\overline{E_{p,n}}$
is an isotropic subspace dual to the column space of $E_{p,n}$.  This
gives a Levi decomposition
\begin{equation} \label{eq28}
P_{p,n}\cong (\GL(n,\bC) \times K_\bC^{p-2n})\ltimes N_{p,n}
\end{equation}
with $N_{p,n}$ its  unipotent radical.  Let 
\begin{equation} \label{eq29}
  \beta_{p,n}\colon  P_{p,n}\to \GL(n,\bC)
\end{equation}
be the group homomorphism defined via quotient by $K_\bC^{p-2n}\ltimes
N_{p,n}$.

\subsubsection{Case I: $q \geq n$} \label{SB11} We refer to
\eqref{eq18}. We set $z_0 = (z_0^+; z_0^-) = (E_{p,n}; E_{q+t,n}) \in
W^+ \oplus W_H^-$. Let $(y^+;y^-) = \pr(z_0)$. We also set $x =
\phi(y^+,y^-) = y^+(y^-)^T$. Then~$y^-$ has full rank $n$.

\subsubsection{Case II: $q < n$} \label{SB12}

We will change the basis in $\bC^p$ such that the first
$r$-coordinates are isotropic and dual to the last $r$-coordinates.

Let $z_0 = (z_0^+;z_0^-)$ with $z_0^+ = I_{p,n}$ and 
$z_0^- = \begin{pmatrix}
I_q & 0\\
iI_q & 0\\
0 & E_{q -2q,n-q} \\
\end{pmatrix}$.  Then $x = \phi(\pr(z_0)) = I_{p,q} \in \fpp^* =
M_{p,q}$. Its stabilizer in $K_\bC^{p,q}$ is
\[
K_x = \Set{(o_1,o_2)\in P_{p,q}\times K_\bC^q|\beta_{p,q}(o_1)
  = o_2\in \rO(q,\bC)}
\]
with Levi subgroup
\[
L = \Set{((k_q,k_{p-2q}),k_q) \in K_\bC^{p,q} | k_q\in K_\bC^q, k_{p-2q} \in
  K_\bC^{p-2q}} \cong \triangle K_\bC^q\times K_\bC^{p-2q}.
\]
We recall that $Q = L \times K_\bC^{t} \times K_\bC'$.

Given $(z^+;z^-) \in Z$, i.e. $\phi(\pr(z^+,z^-)) = x = I_{p,q}$. Then
one can show that up to an action of $Q$,
\begin{equation} \label{eq32}
(z^+;z^-) = (\begin{pmatrix}
I_q & \bfzero \\ 
\bfzero & A_s\\ 
\bfzero & \bfzero \end{pmatrix};
 \begin{pmatrix}
I_q &   \bfzero \\
iI_q & \bfzero\\
 \bfzero & B_s\\
\end{pmatrix}) \in \bcNp \times \bcNn
\end{equation}
where $A_s \in M_{p-2q, n-q}$ and $B_s \in  M_{q-2q,n-q}$.

\def\bfM{{\mathbf{M}}}

Let $\bfM = M_{q,n}\times M_{q,n} \times M_{q,t}$. We define an action
of $Q$ on $\bfM$ by $(r_1, r_2, g) (m_1,m_2,m_3) = (r_1 m_1 g^{-1}, r_1
m_2 g^T, r_1 m_3 r_2^T)$ where $r_1 \in K_\bC^q, r_2 \in
K_\bC^t, g \in K_\bC',$ and  $(m_1,m_2,m_3) \in M_{q,n} \times M_{q,n} \times
M_{q,t}$. The subgroup $K_\bC^{p-2q}$ acts trivially. We
define $\pi \colon \bcN \to \bfM$ by
\[
M_{p,n}\times M_{q,n}\ni (
\begin{pmatrix}
A_1\\
*
\end{pmatrix}
,\begin{pmatrix}
B_1 \\ B_2 
\end{pmatrix}
)  \mapsto (A_1, B_1, A_1B_2^T).
\] 
Here $A_1, B_1 \in M_{q,n}, B_2 \in M_{t,n}$. Note that $\pi$
commutes with the action of $Q$.

Let $m = (I_{q,n},I_{q,n},iI_{q,t}).$ Let $\cM = \pi(Z)$ in
$\bfM$. Using \eqref{eq32}, we deduce that $\cM$ is an $Q$-orbit
generated by $m$.  One can also check that
\[
Z_m = \pi^{-1}(m) =  \Set{
(\begin{pmatrix}I_q & \bfzero \\ \bfzero & A_s\\ \bfzero &
  \bfzero  \end{pmatrix};
 \begin{pmatrix}
I_q & \bfzero \\
iI_q & \bfzero\\
 \bfzero & B_s\\
\end{pmatrix})|(A_s,B_s)\in \overline{\cN_s}} \stackrel{T}{\longrightarrow} \overline{\cN_s}
\]
where $T$ is a bijection which maps the above element in the
parathesis to $(A_s,B_s)$ and $\bcN_s$ is the null cone for pair
\[
(G_s,G'_s) = \left( G^{p-2q,t-q}, G'^{2(n-q)} \right) =
\left( \rO(p-2q,t-q), \Sp(2(n-q),\bR) \right).
\]

Finally, the stabilizer of $m$ in $Q$ is
\[
\begin{split}
  Q_m  \cong & \triangle K_{\bC}^q \times K_{\bC}^{p-2q} 
\times K_{\bC}^{t-q} \times K_{\bC}'^{n-q} \\ 
  \cong & \triangle \rO(q,\bC) \times \rO(p-2q,\bC)\times
  \GL(n-q,\bC)\times \rO(t-q,\bC) \\
  \cong & L \times \GL(n-q,\bC) \times \rO(t-q,\bC).
\end{split}
\]
The subgroup $\triangle K_{\bC}^q = \triangle \rO(q,\bC)$ acts
trivially on $\bcNs$. This proves (A), (B), (C) and (E) of \Cref{S58}.

\subsection{} \label{SB1} We refer to the notation in \Cref{S5}. We
recall $\cD = (\phi \circ \pr)^{-1}(\cO)\cap \cN$. 

We recall $z_0 = (z_0^+;z_0^-) =
  (E_{p,n}, E_{q+t,n}) \in \cN$ in Appendices~\ref{SB11} and~\ref{SB12}.
We write $z_0 = (z_0^+; z_1^-,z_2^-)$ and $\cE  = \pr(\cD)$. 
  
\begin{lemma} \label{LB1}
In both Cases I and II, we have the following statements.
\begin{asparaenum}[(i)]
\item 
  The set $\cD$ is a non-empty open dense in $\cN$. 
\item The set $\cD$ is a single $K_\bC^{p,q}\times K^t_\bC\times K'_\bC$-orbit
  generated by $z_0$.
\item The set $\cE$ is dense in $\pr(\bcN)$ and it is a single
  $K_\bC^{p,q}\times K'_\bC$-orbit generated by $\pr(z_0)$.
\end{asparaenum}
 \end{lemma}

\begin{proof}[Sketch of the proof]
\begin{asparaenum}[(i)]
\item Clearly $z_0 \in \cD$ so $\cD$ is non-empty and open in~$\cN$.
  If $\cN$ is irreducible, then $\cD$ is open dense in $\cN$.  If
  $\cN$ is not irreducible, then $\cD$ is open dense because
  $K^{p,q}_\bC\times K^{t}_\bC$ permutes the irreducible components.

\item Let $w = (z^+;z^-) \in \cD$. By the action of $K_\bC^p$, we may
  assume that $z^+ = z_0^+$. Now~(ii) would follow from the claim that
  $z^-$ is in the $K_\bC^q\times K_\bC^t\times K'_\bC$-orbit of
  $z_0^-$. Indeed it is a case by case elementary computation to show
  that
\begin{equation}\label{eq:cD}
\begin{split}
\cD = & \Set{
\begin{array}{l}
w = (w^+;w_1,w_2)\\
\in M_{p,n} \times M_{q,n} \times M_{t,n}
\end{array}
|
\begin{array}{l} w \in \cN, \rank(w_1) = \min(q,n) \\
\rank(w_1^Tw_1) =\min(n,t).
\end{array}}. 
\end{split}
\end{equation}
Our claim is an application of the Witt's theorem to \eqref{eq:cD}. We
will leave the details to the reader.

\item This follows from (i), (ii) and the definition $\cE = \pr(\cD)$.
\end{asparaenum}
\end{proof}


\def\ty{{\tilde{y}}}
\subsection{} \label{SB7}
{\it Proof of \Cref{lem:geo.I}.}
\begin{inparaenum}[(i)]
\item Without loss of generality, we set $y = \pr(z_0) = (y^+;y^-)$
  and $x = \phi(y) \in \cO$. We claim that $Y = \phi^{-1}(x) \cap \pr(\bcN)
  \subset \cE$.  Let $\ty = (\ty^+,\ty^-)\in Y$.  Since $x$ has rank
  $n$, $y^+$ and $\ty^+$ has rank $n$.

  Note that the column spaces of $x,y^+,\ty^+$ are the same.  Hence
  there is $k'\in \GL(n,\bC) = K'_\bC$ such that $\ty^+ =y^+k'$,
  i.e. $\ty^+\in \cN^+$.  Now $y^+(\ty^-k'^T)^T = \ty^+(\ty^-)^T=x =
  y^+ y^-$. Viewing $y^+$ as an injection, we get $y^- =
  \ty^-k'^T$. Therefore $(\ty^+,\ty^-)\in \cE$ by \eqref{eq:cD} which
  proves the claim.
  
  By \Cref{LB1}~(iii), $\cE$ is a single
  $K_\bC\times K'_\bC$-orbit so $\ty = (k,k') \cdot y \in Y$ for some $k \in
  K_\bC$ and $k' \in K_\bC'$.  We have $x = \phi((k,k')\cdot y) = k
  \cdot \phi(\ty)
  = k\cdot x$ so $k \in K_x$. Hence $Y$ is a single orbit of
  $K_{ x} \times K_\bC'$ and this proves (i).

\item 
 We claim that for every $k = (k^p,k^q)\in K_{x}
\subset K_\bC^{p,q}$, there is a unique $k' \in K'_\bC$, denoted by
$\beta(k)$, such that $(k,k')\in S_y$. It is easy to see that $k
\mapsto \beta(k)$ is a group homomorphism $\beta\colon K_{x}\to
K'_\bC$. The image of $\beta$ is contained in $K'_{x'}$ since
$\beta(k)$ stabilizes $\psi(y) = x'$.

  We recall~\eqref{eq28} that $P_{p,n}$ is the parabolic subgroup in
  $K_\bC^p$ which stabilizes column space of $y^+$. We also
  recall~\eqref{eq29} the projection $\beta_{p,n} \colon P_{p,n} \to
  \GL(n,\bC) = K'_\bC$. It satisfies $k^p\cdot y^+ =
  y^+\beta_{p,n}(k^p) = \beta_{p,n}(k^p)^{-1}\cdot y^+$.

  
We now prove our claim.   We define $\beta(k)
= \beta_{p,n}(k^p)$. Then $(k,\beta(k))$ stabilizes $y^+$. Now 
\[y^+
(y^-)^T  = x =
\phi(y) = \phi((k, \beta(k))\cdot y) = \phi(y^+, (k^q,\beta(k))\cdot y^-)= 
y^+ (k^q y^- \beta(k)^{T})^T
\] as matrices. Since $y^+$ is an injective linear transformation, we
have $k^q y^- \beta(k)^{T} = y^-$, i.e. $(k,\beta(k))\in S_y$.

Next we prove the uniqueness of $k'$ in our claim. Indeed if $(k,a),
(k,b) \in S_y$ where $a, b \in K_\bC' = \GL(n,\bC)$. Then
$k^py^+(k^qy^-a^{T})^T = k^py^+(k^qy^- b^{T})^T$ as matrices
so $y^- a^{T}= y^- b^{T}$. Since $y^-$ has rank $n$, $a = b$. This
proves our claim and (ii).

\smallskip

\noindent (iii) An element of $Z_y$ is of the form $(y^+;y^-,y_2)$.
Since $(y^-,y_2)$ is in the null cone $\bcNn$, $\psi_2(y_2) =
-\psi^-(y^-) = x'$, i.e. $y_2\in \psi_2^{-1}(x')$.  On the other hand,
for any $y_2\in \psi_2^{-1}(x')$, $(y^-,y_2)\in \cN^-$ since $y^-$
already has full rank. We define $T : Z_y \rightarrow Z_{x'}$ by
$(y^+;y^-,y_2)\mapsto y_2$. This is a bijection which satisfies
(iii). \qed
\end{inparaenum}

\subsection{} \label{SB5} {\it Proof of \Cref{L62}.}  The proof of the
lemma involves some elementary but tedious case by case
consideration. The case $n>t$ and $n>q$ are symmetric, so we only
sketch the proof for $q>n>t$.

By \Cref{LB1}(ii), $\cD$ is a single $K_\bC^{p,q} \times K'_\bC \times
K^t_\bC$-orbit. Let $\cD^- = \{ (w_1,w_2) | (w^+;w_1,w_2) \in \cD \}$.
Then $\cD^-$ is the $K_\bC^q \times K_\bC^{t}\times K_\bC'$-orbit of
$z_0^- = E_{q+t,n}$ and $\cD =\cN^+ \times \cD^-$.

\smallskip

Let $c$ denote the co-dimension of $\cN - \cD$ in $\cN$. It is also
the codimension of $\cN^- - \cD^-$ in $\cN^-$. We need to show that $c
\geq 2$. If $q +t > 2n$ so there exists a row of zeros in~$E_{q+t,n}$
and we set $E^*$ to be the $q+t$ by $n$ matrix obtained from the
matrix $E_{q+t,n}$ by interchanging a zero row and with the last
row. If $q+t =2n$, we interchange the $n$-th row with the $(q+t-1)$-th
row. In both cases, let $\cN^*$ denote the $K^q_\bC \times K^t_\bC
\times K_\bC'$-orbit generated by~$E^*$.

We observe that $\cN^*$ is open dense in $\cN^--\cD^-$.




The codimension $c = \dim \cN^- - \dim \cN^* = \dim S_0^* - \dim S_0$
where $S_0$ and $S_0^*$ are the stabilizers of $E_{q+t,n}$ and $E^*$
respectively in $K^q_\bC\times K^t_\bC\times K_\bC'$. One may compute
that
\[
\dim S_0 = \frac{1}{2} \left(n^2-n + q^2-q + t^2 - t \right) + n(n - q
- t + 1).
\]
Similarly $\dim S_0^*$ has the same formula as $\dim S_0$ above except
that we have to reduce $t$ by~$1$.  With these, we
compute that $\dim S_0^* - \dim S_0 = 1 + n - t \geq 2$ as
required. \qed


\begin{thebibliography}{99}
\bibitem[BV]{BV} Barbash and D., Vogan, {\em Unipotent representation
    of complex semisimple groups}, Ann. Math. 121 (1985), 41-110.

\bibitem[BK]{BK} W. Borho and H. Kraft, {\em Uber die
  Gelfand-Kirillov-Dimension.} Math. Ann. 220 (1976), 1-24.

\bibitem[CPS]{CPS} E. Cline, B. Parshall and Leonard Scott, {\em A
    Mackey Imprimitivity Theory for Algebraic Groups}. Math. Z. 182
  (1983), 447-471.


\bibitem[CM]{CM} D.H. Collingwood and W. Mcgovern, {\em Nilpotent
      orbits in semisimple Lie algebras.} Van Nostrand Reihnhold
      Mathematics Series, New York, (1993).

 \bibitem[DKP]{DKP} A. Daszkiewicz, W. Kra\'skiewicz, T. Przebinda,
      {\em Dual pairs and Kostant-Sekiguchi
        correspondence. II. Classification of nilpotent elements},
      Central European Journal of Mathematics 3(3), (2005), 430-474.

\bibitem[DER]{DER} M. Davidson, T. Enright, R. J. Stanke, {\em
    Differential operators and highest weight representations.} Mem.
  Amer. Math. Soc. 94 (1991).

\bibitem[EHW]{EHW} T. Enright, R. Howe and N. Wallach, {\em A
    classification of unitary highest weight modules.}  Representation
  theory of reductive groups (Park City, Utah, 1982), 97--143, Progr.
  Math., 40, Birkh\"auser Boston, Boston, MA, (1983).


\bibitem[GW]{GW} Roe Goodman and Nolan R. Wallach, {\em
    Representations and Invariants of the Classical Groups} Cambridge
  U. Press, 1998; third corrected printing, 2003.
    

\bibitem[D]{Du} M. Duflo, {\em Sur la classification des id\'{e}au
   primitifs dans l'alg\'{e}bre enveloppante d'une alg\'{e}bre de Lie
   semi-simple.} Ann. Math. (2) 105 (1) (1977) 107-120.



\bibitem[Ha]{Ha} R. Hartshorne, {\em Algebraic Geometry.} Graduate
  Texts in Mathematics, 52. Springer-Verlag (1983).



\bibitem[He]{He} H. He, {\em Theta Correspondence in Semistable Range
    I: Construction and Irreducibility.} Commun. Contemp. Math.  {\bf
    2} (2000), no. 2, 255--283.


\bibitem[H1]{H1} R. Howe, {\em Remarks on classical invariant
      theory.} Trans. Amer. Math. Soc. {\bf 313} (1989), no. 2,
    539--570.

\bibitem[H2]{H2} R. Howe {\em Transcending classical invariant theory.}
J. AMS {\bf 2} no. 3 (1989).


\bibitem[H3]{H3} Roger Howe {\em Perspectives on invariant theory:
    Schur duality, multiplicity-free actions and beyond.}  The Schur
    lectures (1992) (Tel Aviv), 1--182, Israel Math. Conf. Proc., 8,
    Bar-Ilan Univ., Ramat Gan, 1995.

  \bibitem[Hu]{Hu} J. E. Humphreys, {\em Linear algebraic groups.}  
  Graduate Texts in Mathematics, 21. Springer-Verlag (1981).

\bibitem[KV]{KV} M. Kashiwara and M. Vergne, {\em On the
    Segal-Shale-Weil representations and harmonic polynomials}.
  Invent. Math. {\bf 44} (1978), no. 1, 1-47.

\bibitem[K]{Kob} T. Kobayashi, {\em Discrete decomposability of the
    restriction of $A_{\frakq}(\lambda)$ with respect to reductive
    subgroups. III. Restriction of Harish-Chandra modules and
    associated varieties.} Invent. Math. 131, no. 2  (1998), 229-256.

\bibitem[K\O]{KO} T. Kobayashi and B. {\O}rsted, {\em Analysis of the
    minimal representation of $\rO(p,q)$ I, II, III.} Adv. Math. 180,
    no. 2 (2003), 486-595. 

\bibitem[Ko]{Ko} B. Kostant, {\em Lie group representations on
    polynomial rings.} Am. J. Math. 85, (1963), 327-404.

\bibitem[Li1]{Li} J.-S. Li, {\em Singular unitary representations of
classical groups.} Invent. Math. {\bf 97} (1990), 237-255.

\bibitem[Lo]{Lo} H. Y. Loke, {\em Howe quotients of unitary
    characters and unitary lowest weight modules.} Representation
  Theory. 10 (2006), 21-47.

\bibitem[LM]{LM} H. Y. Loke, J.-j. Ma, {\em Invariants and
    $K$-spectrums of local theta lifts}. Preprint arXiv:1302.1031 (2013).

\bibitem[LMT]{LMT} H. Y. Loke, J.-j. Ma and U. Tang {\em Transfer of
    $K$-types and local theta lifts of unitary characters and lowest
    weight modules.} To appear in the Israel of Math. arXiv:1207.6454.

\bibitem[NOT]{NOT} K. Nishiyama, H. Ochiai and K. Taniguchi, {\em
    Bernstein degree and associated cycles of Harish-Chandra
    modules---Hermitian symmetric case.} Nilpotent orbits, associated
  cycles and Whittaker models for highest weight
  representations. Ast\'{e}risque No. 273 (2001), 13--80.

\bibitem[NOZ]{NOZ} K. Nishiyama, H. Ochiai and C. B. Zhu, {\em Theta
     lifting of nilpotent orbits for symmetric pairs,}
   Trans. Amer. Math. Soc., 358, (2006), 2713-2734.

\bibitem[NZ]{NZ} K. Nishiyama and C. B. Zhu, {\em Theta lifting of
    unitary lowest weight modules and their associated cycles.} Duke
  Math. J. 125, no. 3 (2004), 415-465.


\bibitem[Oh]{Ohta} T. Ohta, {\em The closure of nilpotent orbits in the
    classical symmetric pairs and their singularities},
  Tohoku Math. J. 43, No. 2, 161--211, (1991), Tohoku University.


\bibitem[P]{P} Tomasz Przebinda, {\em Characters, dual pairs, and
    unitary representations.}   Duke Math. J. Volume 69, Number 3
  (1993), 547-592.

\bibitem[T1]{T} Peter E. Trapa, {\em Some small unipotent
    representations of indefinite orthogonal groups.} Journal of
    Functional Analysis {\bf 213} (2004), 290-320.

  \bibitem[T2]{T2} Peter E. Trapa, {\em Special unipotent
      representations and the Howe correspondence} University of
    Aarhus Publication Series, 47 (2004), 210-230.

\bibitem[V2]{V2} David A. Vogan, {\em Unitary representations of
      reductive Lie groups.} Annals of Mathematics Studies, vol. 118,
    Princeton University Press, Princeton, N. J., (1987).

\bibitem[V3]{VoAss} D. Vogan, {\em Associated varieties and unipotent
      representations,} Harmonic Analysis on Reductive Lie Groups.
      Progr. in Math. 101. Birkhauser  (1991), 315-388.

\bibitem[Ya]{Ya} H. Yamashita, {\em Cayley transform and generalized
    Whittaker models for irreducible highest weight modules.}
  Ast\'{e}risque 273 (2001), 81-137.

\bibitem[Yn]{Y} Liang Yang, {\em On quantization of spherical nilpotent
  orbits of $\fgg$-height 2}, HKUST Thesis (2010).

\bibitem[ZH]{ZH} C.-B. Zhu and J.-S. Huang, {\em 
    On certain small representations of indefinite orthogonal groups.}
  Represent. Theory 1 (1997), 190-206. 

\bibitem[Z]{Zhu} C.-B. Zhu, {\em Invariant distributions of classical
groups.} Duke Math. J. {\bf 65} (1992), 85-119.
\end{thebibliography}
\end{document}